\documentclass{article}

\usepackage[top=2.5cm, bottom=2.5cm, left=3cm, right=3cm]{geometry}   
\usepackage{amssymb}
\usepackage{amsthm}
\usepackage{amsmath}
\usepackage{graphicx}
\usepackage{graphics}
\usepackage{amssymb}
\usepackage{multirow} 
\usepackage{lipsum,multicol}
\usepackage{amsmath}
\usepackage{tikz,mathpazo}
\usepackage{subfigure}
\usepackage{epsfig}
\usepackage{epstopdf}
\usepackage[numbers]{natbib} 
\usepackage{url,comment}
\usepackage{booktabs}

\newcommand{\inner}[1]{\left\langle #1 \right\rangle}
\newcommand{\norm}[1]{\left\Vert #1\right\Vert}
\newcommand{\bb}[1]{\mathbb{#1}}

\newcommand{\ca}[1]{\mathcal{#1}}
\newcommand{\tr}[0]{\mathrm{tr}}

\newcommand{\tX}[0]{\tilde{X}}

\newcommand{\Diag}[0]{\mathrm{Diag}}

\newcommand{\ff}{_{\mathrm{F}}}
\newcommand{\fs}{^2_{\mathrm{F}}}
\newcommand{\tp}{^\top}

\newcommand{\Apen}[1]{\left( \frac{3}{2} I_p - \frac{1}{2} {#1}\tp {#1} \right)}

\newcommand{\Xk}{{X_{k} }}

\newcommand{\Xkp}{{X_{k+1} }}

\newcommand{\PT}[0]{{J}_{X}}

\newcommand{\BOM}{\overline{\Omega}}
\newcommand{\BOMs}{\overline{\Omega}_{1/6}}
\newcommand{\BOMt}{\overline{\Omega}_{1/12}}
\newcommand{\BOMtiny}{\overline{\Omega}_{1/24}}

\newcommand{\grad}{{\mathit{grad}\,}}
\newcommand{\hess}{{\mathit{hess}\,}}

\newcommand{\ProjS}{{\ca{P}_{\ca{S}_{n,p}}}}

\newcommand{\Expen}{{ExPen}}

\newcommand{\Expens}{{ExPen}\,}

\newtheorem{theo}{Theorem}[section]
\newtheorem{lem}[theo]{Lemma}
\newtheorem{prop}[theo]{Proposition}

\newtheorem{coro}[theo]{Corollary}

\newtheorem{defin}[theo]{Definition}
\newtheorem{rmk}[theo]{Remark}
\newtheorem{assumpt}[theo]{Assumption}

\usepackage{caption}
\usepackage{algorithm}
\usepackage{algpseudocode}
\usepackage{longtable}
\usepackage{appendix}

\numberwithin{equation}{section}

\title{Solving Optimization Problems over the Stiefel Manifold by Smooth Exact Penalty Functions}

\author{
	{
		Nachuan Xiao\thanks{ The Institute of Operations Research and Analytics, National University of Singapore, Singapore. (email: xnc@lsec.cc.ac.cn).}~
		~Xin Liu\thanks{
			State Key Laboratory of Scientific and Engineering Computing, Academy of Mathematics and Systems Science, Chinese Academy of Sciences, and University of Chinese Academy of Sciences, China (email: liuxin@lsec.cc.ac.cn). Research is supported in part by the National Natural Science Foundation of 
			China (No. 12125108, 11971466, 11991021), Key Research Program of Frontier 
			Sciences, Chinese Academy of Sciences (No. ZDBS-LY-7022)}
	}
}

\begin{document}
	\maketitle
	
	\begin{abstract}
		In this paper, we present a novel penalty model called \Expens for optimization over the Stiefel manifold. Different from existing penalty functions for orthogonality constraints, \Expens adopts
		a smooth penalty function without using any first-order derivative of the objective function. We show that all the first-order stationary points of \Expens with a sufficiently large penalty parameter are either feasible, namely, are the first-order stationary points of the original optimization problem, or far from the Stiefel manifold. Besides, the original problem and \Expens share the same second-order stationary points. Remarkably, the exact gradient and Hessian of \Expens are easy to compute. As a consequence, abundant algorithm resources in unconstrained optimization can be applied straightforwardly to solve \Expens.
	\end{abstract}

	\section{Introduction}
In this paper, we consider the following optimization problem
\begin{equation}
	\label{Prob_Ori}
	\tag{OCP}
	\begin{aligned}
		\min_{X \in \bb{R}^{n\times p}} \quad &f(X)\\
		\text{s.t.} \quad & X\tp X = I_p,
	\end{aligned}
\end{equation}
where $I_p$ denotes the $p\times p$ identity matrix, and $f: \bb{R}^{n\times p} \mapsto \bb{R}$ satisfies the
following assumption throughout this paper.
\begin{assumpt}{\bf Blank assumptions on $f$}
	\label{Assumption_local_f}
	
	\begin{enumerate}
		\item $f$ and $\nabla f$ are locally Lipschitz continuous in $\bb{R}^{n\times p}$ \footnote{A mapping $T: \bb{R}^{n\times p} \to \bb{R}^{m}$ is locally Lipschitz continuous over $\bb{R}^{n\times p}$ if for any $X_0 \in \bb{R}^{n\times p}$, there exists a constant $M$ and $\delta > 0$ such that for any $X \in \bb{R}^{n\times p}$ satisfying $\norm{X - X_0}\ff \leq \delta$, it holds that $\norm{T(X) - T(X_0)} \leq M\norm{X - X_0}\ff$.}.
	\end{enumerate}
\end{assumpt}

The feasible region of the orthogonality constraints $X\tp X = I_p$ is the Stiefel manifold embedded in the $n\times p$ real matrix space, denoted by $ \ca{S}_{n,p} := \{X \in \bb{R}^{n\times p}| X\tp  X = I_p \}$. We also call it as the Stiefel manifold for brevity. Optimization problems with orthogonality constraints have wide applications in statistics  \cite{pearson1901liii,fischler1981random}, scientific computation \cite{Kohn1965Self,liu2015analysis}, image processing \cite{bansal2018can} and many other related areas \cite{han2015unsupervised,li2018generalized,zhang2018feature}. Interested readers could refer to some recent works \cite{gao2019parallelizable,huang2019riemannian,xiao2021penalty,wang2020distributed}, a recent survey \cite{hu2020brief}, and several books \cite{Absil2009optimization,boumal2020introduction} for details.

\subsection{Motivation}

Optimization over the Stiefel manifold, which is a smooth and compact Riemannian manifold, has been discovered to enjoy a close relationship with unconstrained optimization. However, developing optimization approaches over the Stiefel manifold is inherently complicated by the nonconvexity of the manifold.
Various existing {\it unconstrained optimization approaches}, i.e. the approaches for solving nonconvex unconstrained optimization problems, can be extended to their Riemannian versions by the local diffeomorphisms between the Stiefel manifold and Euclidean space.
The approaches, called {\it Riemannian optimization approaches} for brevity hereinafter, include gradient descent with line-search \cite{abrudan2008steepest,Absil2009optimization,wen2013feasible,jiang2015framework,wang2020multipliers}, conjugate gradient methods \cite{abrudan2009conjugate}, Riemannian accelerated gradient method \cite{zhang2018r,zhang2018towards,siegel2019accelerated,criscitiello2020accelerated}, Riemannian adaptive gradient methods \cite{becigneul2018riemannian}, etc. 
With the frameworks and geometrical materials described in \cite{Absil2009optimization}, theoretical results of these Riemannian optimization approaches have been established by following almost the same proof techniques as their unconstrained prototypes. These results include the global convergence, local convergence rate, worst-case complexity, saddle-point-escaping properties, etc., see \cite{criscitiello2019efficiently,sun2019escaping,zhou2019faster,ahn2020nesterov,criscitiello2020accelerated,han2020escape,hou2020analysis} for instances. 

The Riemannian optimization approaches usually consist of two fundamental parts.
The first one is the so-called ``retraction" which maps a point from the tangent space to the manifold. Retractions can be further categorized into two classes: the geodesic-like retractions
and the projection-like ones. The former ones require to calculate the geodesics along the manifold and hence are expensive. The latter ones enjoy relatively lower computational cost, 
but as demonstrated in various existing works \cite{gao2019parallelizable,xiao2020class,xiao2021penalty}, computing the projection-like retractions is still more expensive than matrix-matrix multiplication.
The second part is called ``parallel transport" which
moves a tangent vector along a given curve on a Stiefel manifold ``parallelly".
The purpose of parallel transport is to design the manifold version of some advanced unconstrained optimization approaches, such as conjugate gradient methods or gradient methods with momentum.
However, as illustrated in \cite{Absil2009optimization}, computing the parallel transport on Stiefel manifold is equivalent to finding a solution to a differential equation, which is definitely impractical in computation. 
To this end, the authors of \cite{Absil2009optimization} have proposed the concept of vector transport, which can be regarded as an approximation to parallel transport, hence is computationally affordable. Unfortunately, due to the approximation error introduced by vector transports, analyzing the convergence properties of Riemannian optimization algorithms is challenging and usually cannot directly follow the their unconstrained origins, see \cite{huang2019riemannian} for instances.  As illustrated in various existing works \cite{zhang2018r,zhang2018towards,criscitiello2019efficiently,zhou2019faster,ahn2020nesterov,criscitiello2020accelerated}, both parallel transports and geodesics play an essential role in establishing convergence properties. It is still difficult to verify whether their theoretical convergence properties is valid when these approaches are built by retractions and vector transports. 

To avoid computing the retractions, parallel transports, or vector transports to the Stiefel manifold, some approaches
aim to find  smooth mappings from the Euclidean space to the Stiefel manifold, 
which directly reformulates \ref{Prob_Ori} to unconstrained optimization. Among them, \cite{lezcano2019cheap,lezcano2019trivializations} construct equivalent unconstrained problems for \ref{Prob_Ori} by exponential function for square matrices. To efficiently compute the matrix exponential, they apply the iterative approach proposed by \cite{higham1994APA}. Therefore, their approaches require $\ca{O}(n^3)$ flops in each iteration for computing the matrix exponential and thus are computationally expensive in practice.  
Inspired by \cite{lezcano2019cheap,lezcano2019trivializations}, several recent works use Cayley transformation \cite{helfrich2018orthogonal,maduranga2019complex,figueroa2020transportless,li2020efficient} to avoid computing the matrices exponential. 
These approaches require computing the inverse of $n\times n$ matrices in each iterate, which still requires $\ca{O}(n^3)$ flops in general. Furthermore, calculating the derivatives through the exponential mapping or the Cayley transformation
	can be more costly. 
	In particular, as illustrated in the numerical experiments in \cite{figueroa2020transportless,ablin2022fast}, when applying nonlinear conjugate gradient methods, the computational time of these approaches is usually much higher than existing Riemannian conjugate gradient approaches.

Recently, some infeasible approaches
have been verified to be efficient in solving optimization problems
over the Stiefel manifold. They 
utilize a completely different angle with existing Riemannian optimization approaches.
Based on the framework of the augmented Lagrangian method (ALM) 
\cite{hestenes1969multiplier,powell1969method,nocedal2006numerical,bertsekas2014constrained}, the authors of 
\cite{gao2019parallelizable} have proposed the proximal linearized augmented Lagrangian method
(PLAM) and its column-wise normalization version (PCAL) for \eqref{Prob_Ori}. 
Both PLAM and PCAL update the multipliers corresponding to the orthogonality constraints by a closed-form expression. Additionally, \cite{ablin2022fast} have proposed the landing algorithm, which follows a two-step alternative updating framework. 
%
Inspired by the closed-form updating scheme in PLAM and PCAL, the authors of \cite{xiao2020class} have proposed an exact penalty function named PenC,
\begin{equation*}
	\min_{ \norm{X}\ff \leq K} \quad h_{PenC}(X) = f(X) - \frac{1}{2} \inner{\Phi(X\tp \nabla f(X)), X\tp X - I_p} + \frac{\beta}{4} \norm{X\tp X - I_p}\fs,
\end{equation*} 
where $K \geq \sqrt{p}$ is a prefixed constant and $\Phi$ is the symmetrization operator defined as  
$$\Phi(M) := \frac{1}{2}(M + M\tp). $$ 
In \cite{xiao2020class}, the authors have illustrated the equivalence between \ref{Prob_Ori} and PenC, which further proposed the corresponding infeasible first-order and second-order methods PenCF and PenCS, respectively. 
Moreover, successive works \cite{xiao2020l21,hu2020anefficiency} have illustrated that PenC could be extended to objective function with special structures. 
The above-mentioned penalty-function-based approaches are verified to enjoy high efficiency and scalability due to avoiding retractions or parallel transport to the Stiefel manifold. However, their penalty functions involve the first-order derivatives of the original objective, which leads to two limitations. Firstly, the smoothness of
the penalty function requires higher-order smoothness of the original objective function. Secondly, calculating an exact gradient of these penalty functions is usually expensive in practice. As a result, many existing unconstrained optimization approaches cannot be directly applied to minimize these penalty functions.

\subsection{Contributions}
The contributions of this paper can be summarized as the following
two folds.

\paragraph{A novel penalty function}
We propose a novel penalty function
\begin{equation}
	\label{Penalty_function}
	h(X):= f\left( X\Apen{X}\right) + \frac{\beta}{4} \norm{X\tp X - I_p}\fs,
\end{equation}
and construct the following unconstrained optimization problem which is abbreviated as ExPen.
\begin{equation}
	\label{Prob_Pen}
	\tag{ExPen}
	\min_{X \in \bb{R}^{n\times p}}\quad h(X).
\end{equation}
With a sufficiently large penalty parameter $\beta$, we illustrate that any first-order stationary point (FOSP) of \ref{Prob_Pen} is either feasible and hence a FOSP of \ref{Prob_Ori}, or far away from the Stiefel manifold. Besides, we prove that any eigenvalue of the Riemannian Hessian at any first-order stationary point $X \in \ca{S}_{n,p}$  is an eigenvalue of $\nabla^2 h(X)$. Then we show that any second-order stationary point (SOSP) of \ref{Prob_Pen} is an SOSP of  \ref{Prob_Ori}. We call the above two relationships the {\it first-order relationship} and {\it second-order relationship}, respectively, for brevity.  These two relationships imply that \Expens can be regarded as an exact penalty function.

\paragraph{A universal tool}
The exact penalty model \ref{Prob_Pen} builds up a bridge between various existing unconstrained optimization approaches
and \ref{Prob_Ori}. 
Moreover, those rich theoretical results of unconstrained optimization approaches can be directly applied in solving
\ref{Prob_Ori}. 
In particular, some newly developed techniques for unconstrained optimization can be
extended to solve optimization over the Stiefel manifold through \Expen.
We use the nonlinear conjugate gradient method as an example. 
It is difficult to find a compromise between computational efficiency and theoretical guarantee
if we adopt Riemannian optimization approaches to achieve this extension.
Preliminary numerical experiments illustrates that \ref{Prob_Pen} yields direct and efficient of nonlinear conjugate gradient solver from SciPy package.

\subsection{Notations}\label{Sec:notation}
In this paper, the Euclidean inner product of two matrices $X, Y\in \bb{R}^{n\times p}$ is defined as $ \inner{X, Y}=\tr(X\tp Y)$,
where $\tr(A)$ is the trace of the square matrix $A$. Besides, 
$\norm{\cdot}_2$ and $\norm{\cdot}\ff$ represent the $2$-norm and the Frobenius norm, respectively. 
The notations $\mathrm{diag}(A)$ and $\Diag(x)$
stand for the vector formed by the diagonal entries of matrix $A$,
and the diagonal matrix with the entries of $x\in\bb{R}^n$ to be its diagonal, respectively. 
We denote the smallest eigenvalue of $A$ by $\lambda_{\mathrm{min}}(A)$. 
  We set the Riemannian metric on Stiefel manifold as the metric inherited from the standard inner product in $\bb{R}^{n,p}$. We set
$\ca{T}_X$ as the tangent space of Stiefel manifold at $X$, which can be expressed as 
\begin{equation*}
	\ca{T}_X := \{D \in \bb{R}^{n\times p} \mid \Phi(D\tp X) = 0 \},
\end{equation*}
while $\ca{N}_X$ is denoted as the normal space of Stiefel manifold at $X$,
\begin{equation*}
	\ca{N}_X := \{D \in \bb{R}^{n\times p} \mid D = X\Lambda, \Lambda  = \Lambda\tp\}.
\end{equation*}
And $\grad f(X)$ denotes the Riemannian gradient of $f$ at $X \in \ca{S}_{n,p}$ in Riemannian metric that is inherited from the Euclidean metric, namely, 
\begin{equation*}
	\grad f(X) := \nabla f(X) - X \Phi(X\tp \nabla f(X)).
\end{equation*}
Besides, we uses $\nabla^2 f(X)[D]$ to represent the Hessian-matrix product. The Riemannian Hessian of $f$ at $X \in \ca{S}_{n,p}$ in Euclidean measure is denoted as $\hess f(X): \ca{T}_X \to \ca{T}_X$, whose bilinear form
can be written as
\begin{equation*}
	\inner{D_1, \hess f(X) [D_2] } := \inner{D_1, \nabla^2 f(X) [D_2] - D_2 \Phi(X\tp \nabla f(X))}, \quad \forall D_1, D_2 \in \ca{T}_X . 
\end{equation*}
Finally, $\ProjS(X) = UV\tp$ denotes the orthogonal projection to Stiefel manifold, where $X = U\Sigma V\tp$ is the economic SVD of $X$
with $U\in \ca{S}_{n,p}$, $V\in \ca{S}_{p,p}$ and $\Sigma$
is $p\times p$ diagonal matrix with the singular values of
$X$ on its diagonal.

\subsection{Organization}
The rest of this paper is organized as follows. In Section 2, we present several preliminaries and useful lemmas. Then we explore the first-order and second-order relationships between \ref{Prob_Ori} and \ref{Prob_Pen}, respectively, in Section 3. We show how to solve \ref{Prob_Ori} through unconstrained optimization approaches by an illustrative example in Section 4 and draw a brief conclusion in the last section.

\section{Preliminaries}
In this section, we provide several preliminary properties of \ref{Prob_Pen}. We first introduce the definitions, assumptions and define several constants. Then we give some preliminary properties of \ref{Prob_Pen}. Finally, we present the computational complexity of calculating the derivatives of \ref{Prob_Pen}. 
\subsection{Definitions}
The first-order optimality condition of problem \ref{Prob_Ori} can be written as
\begin{defin}{\cite{Absil2009optimization}}\label{Defin_FOSP}
	Given a point $X\in \ca{S}_{n,p}$, we call $X$ a first-order stationary point of \ref{Prob_Ori} if $\grad f(X) = 0$. 
\end{defin}
According to \cite{Gao2018New}, any $X \in \bb{R}^{n\times p}$ is a first-order stationary point of \ref{Prob_Ori} if and only if it satisfies 
\begin{equation*}
	\left\{\begin{aligned}
		\nabla f(X) - X\Phi(X\tp \nabla f(X)) &= 0;\\
		X\tp X &= I_p.
	\end{aligned}\right.
\end{equation*}

Next, we present the definition of the second-order optimality condition of \ref{Prob_Ori}. 
\begin{defin}
	\label{Defin_SOSP}
	Given a point $X\in \ca{S}_{n,p}$, if $f$ is twice-differentiable, $X$ is a first-order stationary point of \ref{Prob_Ori} and 
	\begin{equation*}
		\inner{D, \hess f(X)[D]} \geq 0,
	\end{equation*}
	holds for any $D \in \ca{T}_X$, then we call $X$ a second-order stationary point of \ref{Prob_Ori}. 
\end{defin}

Besides, we present the definitions of first-order and second-order optimality conditions of \ref{Prob_Pen}.  Given a point $X\in\bb{R}^{n\times p}$, we say $X$ is a first-order stationary point of a differentiable function $h: \bb{R}^{n\times p} \to \bb{R}$ if and only if $\nabla h(X) = 0$. And when $h$ is twice-order differentiable, $X$ is a second-order stationary point of $h$ if and only if $X$ is a first-order stationary point of $h$ and 
\begin{equation}
	\label{Eq_SOSP_h}
	\inner{\nabla^2 h(X)[D], D} \geq 0, \quad \forall D \in \bb{R}^{n\times p}. 
\end{equation}

Next we present the definitions of {\L}ojasiewicz  inequality  \cite{lojasiewicz1961probleme,lojasiewicz1963propriete} in the following, which coincide with the definitions in \cite{bolte2014proximal}. 
\begin{defin}
	
	Let $f$ be a differentiable function. Then $f$ is said to have the (Euclidean) {\L}ojasiewicz gradient inequality at $X \in \bb{R}^{n\times p}$ if and only if there exists a neighborhood $U$ of $X$, and constants $\theta \in (0,1]$, $C>0$, such that for any $Y \in U$,
	\begin{equation*}
		\norm{\nabla  f(Y)}\ff  \geq C |f(Y) - f(X)|^{1-\theta}.
	\end{equation*}
\end{defin}

Besides, we present the definitions of Riemannian {\L}ojasiewicz  inequality \cite{hosseini2015convergence}. 
\begin{defin}
	
	Let $f$ be a differentiable function. Then $f$ is said to have the Riemannian {\L}ojasiewicz gradient inequality at $X \in \ca{S}_{n,p}$ if and only if there exists a neighborhood $U \subset \ca{S}_{n,p}$ of $X$, and constants $\theta \in (0,1]$, $C>0$, such that for any $Y \in U$,
	\begin{equation*}
		\norm{\grad  f(Y)}\ff  \geq C |f(Y) - f(X)|^{1-\theta}.
	\end{equation*}
\end{defin}
The constant $\theta$ is usually named as {\L}ojasiewicz exponent in the gradient inequality.

\subsection{Assumptions}
In this subsection, we present some additional assumptions on the objective function $f$ in \ref{Prob_Ori}, which are usually optional throughout this paper. 
Before presenting these additional assumptions using
in some parts of this paper, we first define some set and operators.
\begin{itemize}
	\item $\Omega := \{ X \in \bb{R}^{n\times p} \mid \norm{X}_2 \leq 1+ \frac{1}{12} \}$;
	\item $\BOM_r :=\{ X \in \bb{R}^{n\times p} \mid 
	\norm{X\tp X - I_p}\ff \leq r\}$;
	\item $G(X):= \nabla f\left( Y \right) \left\vert_{Y = X\Apen{X}} \right. $;
	\item $\ca{H}(X) := \nabla^2 f(Y)\left\vert_{Y = X\Apen{X}} \right.$;
	\item $\PT(D) := D\Apen{X} - X \Phi(D\tp X)$;
	\item $g(X) := f\left(  X\Apen{X} \right)$.
\end{itemize}
Clearly, we have $\BOMt \subset \BOMs\subset \Omega$.
In addition, we present several constants for the theoretical analysis of \ref{Prob_Pen}. 
\begin{itemize}
	\item $M_0 := \sup_{X \in \Omega  } ~  f(X) - \inf_{X \in \Omega } f(X)$;
	\item $M_1 := \sup_{X \in \Omega } ~ \norm{G(X)}\ff $;
	\item $M_2 := \sup_{X, Y \in \Omega, X\neq Y  } ~\frac{\norm{\nabla g(X) - \nabla g(Y)}\ff}{\norm{X-Y}\ff} $;
	\item $\bar{\beta} := \max\{  12M_1, 6M_2 \}$.
\end{itemize}
Clearly, parameters $M_0$, $M_1$ and $M_2$ are well-defined constants
independent of the penalty parameter $\beta$.

\begin{assumpt}{\bf The global Lipschitz continuity of $f$}
	\label{Assumption_global_f}
	
	$f$ is globally Lipschitz continuous in $\bb{R}^{n\times p}$.
\end{assumpt}

Although Assumption \ref{Assumption_global_f} is restrictive, it is optional in our theoretical analysis. 
	It will be specifically mentioned where applicable.
%
Moreover, under Assumption \ref{Assumption_global_f}, we define several additional constants for \ref{Prob_Pen},
\begin{itemize}
	\item $\hat{M}_1 := \sup_{X \in \bb{R}^{n\times p} } ~\norm{G(X)}\ff$;
	\item $\hat{M}_2 := \sup_{X,Y \in \bb{R}^{n\times p}, X\neq Y } ~ \frac{\norm{\nabla g(X) - \nabla g(Y)}\ff}{\norm{X-Y}\ff} $;
	\item $\hat{\beta} := \max\{  12\hat{M}_1, 6\hat{M}_2 \}$.
\end{itemize}
We emphasize that
parameters $\hat{M}_1$ and $\hat{M}_2$  are independent with the penalty parameter $\beta$. Besides,
it follows from Assumption \ref{Assumption_global_f} that $\hat{M}_1 \geq M_1$ and $\hat{M}_2 \geq M_2$.

Furthermore, when we analyze the hessian at $h(X)$, the objective function $f$ in \ref{Prob_Ori} should be twice-differentiable. As a results, in some cases we assume the objective function $f$ be twice differentiable.
\begin{assumpt}{\bf The second-order differentiability of $f$}
	\label{Assumption_second_order}
	
	$\nabla^2 f(X)$ exists at every $X \in \bb{R}^{n\times p}$.  
\end{assumpt}

In the rest of this subsection, we present several useful lemmas for further use. We first show that $\PT$ is the Jacobian of the mapping $X\mapsto X\Apen{X}$ in the following lemma.  
\begin{prop}
	\label{Prop_diff_A}
	For any $X, Y \in \bb{R}^{n\times p}$, let $D = Y-X$, we have
	\begin{equation*}
		\norm{Y\Apen{Y} - X \Apen{X} - \PT(D) }\ff =\ca{O}(\norm{D}\fs). 
	\end{equation*}
	Besides, 
	\begin{equation*}
		\norm{Y \Apen{Y} - X \Apen{X} - \PT(D) - \left[ D \Phi(D\tp X) + \frac{1}{2} X D\tp D \right]}\ff
		= \ca{O}(\norm{D}\ff^3).
	\end{equation*}
\end{prop}

\begin{proof}
	Let $D = Y-X$, from the expression of $Y \Apen{Y}$ we can conclude that 
	\begin{equation*}
		\begin{aligned}
			\begin{aligned}
				& Y \Apen{Y}\\
				={}& X \Apen{X} + D \left( \frac{3}{2} I_p - \frac{1}{2} X\tp X \right) - X \Phi(X\tp D) - D\Phi(D\tp X) - \frac{1}{2}XD\tp D -\frac{1}{2}DD\tp D \\
				={}& X \Apen{X} + \PT(D) - D\Phi(D\tp X) - \frac{1}{2} XD\tp D -\frac{1}{2}DD\tp D,
			\end{aligned}
		\end{aligned}
	\end{equation*}
	and thus complete the proof. 
\end{proof}

In the following Lemma, we present the expression of $\nabla h(X)$:
\begin{prop}
	\label{Le_gradient_h}
	For any $X \in \bb{R}^{n\times p}$, 
	\begin{equation*}
		\nabla h(X) = G(X) \Apen{X} - X \Phi(X\tp G(X)) + \beta X(X\tp X - I_p).
	\end{equation*}
\end{prop}
\begin{proof}
	First, we aim to prove that the linear mapping $\PT$ is self-adjoint for any $X \in \bb{R}^{n\times p}$. From the expression of $\PT$, for any $Z, W \in \bb{R}^{n\times p}$ , we then obtain
	\begin{equation}
		\label{Eq_JA_adjoint}
		\begin{aligned}
			&\inner{\PT(W), Z} 
			= \inner{W \left(\frac{3}{2}I_p - \frac{1}{2}X\tp X\right) - X \Phi(X\tp W), Z}\\
			={}& \tr\left( Z\tp W  \left(\frac{3}{2}I_p - \frac{1}{2}X\tp X\right) \right) - \tr\left( Z\tp X \Phi(X\tp W) \right)\\
			\overset{(i)}{=}{}& \tr\left( W\tp Z  \left(\frac{3}{2}I_p - \frac{1}{2}X\tp X\right) \right) - \tr\left( W\tp X \Phi(X\tp Z) \right)\\
			={}& \inner{Z \left(\frac{3}{2}I_p - \frac{1}{2}X\tp X\right) - X \Phi(X\tp Z), W}
			= \inner{\PT(Z), W}.
		\end{aligned}
	\end{equation}
	Here $(i)$ follows the fact that $\tr \left( A B \right) = \tr\left(A\tp B\right)$ holds for any square matrix $A$ and any symmetric matrix $B$. 
	By Proposition \ref{Prop_diff_A}, for any $X, Y \in \bb{R}^{n\times p}$, let $D = Y-X$, we have
	\begin{equation*}
		\begin{aligned}
			&f\left( Y\Apen{Y} \right)- f\left( X\Apen{X} \right)\\
			={}& \inner{G(X), \PT(D) } + \ca{O}(\norm{D}\fs)
			= \inner{D, \PT(G(X)) } + \ca{O}(\norm{D}\fs)\\
			={}& \inner{D, G(X) \Apen{X} - X \Phi(X\tp G(X) ) } + \ca{O}(\norm{D}\fs),
		\end{aligned}
	\end{equation*} 
	which illustrates that 
	\begin{equation*}
		\nabla g(X) = G(X) \Apen{X} - X \Phi(X\tp G(X)). 
	\end{equation*}
	Then from the fact that $h(X) = g(X) + \beta \norm{X\tp X - I_p }\fs$, we could conclude that 
	\begin{equation*}
		\nabla h(X) = G(X) \Apen{X} - X \Phi(X\tp G(X)) + \beta X(X\tp X - I_p),
	\end{equation*}
	and complete the proof. 
\end{proof}

We can conclude from the definition of $M_1$ and Proposition \ref{Le_gradient_h} that $\norm{\nabla g(X) }\ff \leq 2M_1$ for any $X \in \Omega$. Besides, from the expression of $h(X)$ illustrated in Lemma \ref{Le_gradient_h}, we can conclude that $\nabla h(X) = \grad f(X)$ holds for any $X \in \ca{S}_{n,p}$. Furthermore,  the following proposition illustrates the expression of $\nabla^2 h(X)$. 
\begin{prop}
	\label{Prop_Hessian}
	Suppose $f(X)$ satisfies the conditions in Assumption \ref{Assumption_second_order}, then 
	\begin{equation*}
		\nabla^2 g(X)[D] = \PT  \left(\ca{H}(X)[\PT(D)]\right)- D \Phi(X\tp G(X)) - X \Phi(D\tp G(X)) - G(X) \Phi(D\tp X).
	\end{equation*}
	Moreover, 
	\begin{equation*}
		\nabla^2 h(X)[D] = \nabla^2 g(X) [D] + \beta(2 X\Phi(X\tp D) +  D (X\tp X - I_p)).
	\end{equation*}
\end{prop}
\begin{proof}
	As illustrated in \eqref{Eq_JA_adjoint} from Lemma \ref{Le_gradient_h}, the mapping $\PT$ is self-adjoint for any $X \in \bb{R}^{n\times p}$.  
	Then by Proposition \ref{Prop_diff_A}, for any $Y \in \bb{R}^{n\times p}$ and let $D = Y-X$, we have
	\begin{equation*}
		\begin{aligned}
			&f\left( Y \Apen{Y} \right) - f\left( X \Apen{X} \right) \\
			={}& f\left( X \Apen{X} + \PT(D) - \left[D \Phi(D\tp X) + \frac{1}{2} X D\tp D \right] \right)  - f\left( X \Apen{X} \right) + \ca{O}(\norm{D}\ff^3) \\
			={}& \inner{G(X), \PT(D) - \left[D \Phi(D\tp X) + \frac{1}{2} X D\tp D \right] } + \frac{1}{2}\inner{\ca{H}(X)[\PT(D)] , \PT(D)} + \ca{O}(\norm{D}\ff^3)\\
			={}& \inner{G(X), \PT(D)} - \inner{G(X), D \Phi(D\tp X) + \frac{1}{2} X D\tp D} + \frac{1}{2}\inner{\ca{H}(X)[\PT(D)] , \PT(D)} + \ca{O}(\norm{D}\ff^3)\\
			={}& \inner{D, \nabla g(X)} - \inner{\Phi(D\tp G(X)), \Phi(D\tp X)} - \frac{1}{2} \inner{D\tp D, X\tp G(X)} \\
			&+ \frac{1}{2} \inner{\ca{H}(X)[\PT(D)] , \PT(D)} + \ca{O}(\norm{D}\ff^3).\\
		\end{aligned}
	\end{equation*}
	Therefore, the hessian of $g(X)$ can be expressed as 
	\begin{equation*}
		\nabla^2 g(X)[D] = \PT  \left(\ca{H}(X)[\PT(D)]\right)- D \Phi(X\tp G(X)) - X \Phi(D\tp G(X)) - G(X) \Phi(D\tp X).
	\end{equation*}
	
	Moreover, since 
	\begin{equation*}
		\begin{aligned}
			&\norm{Y\tp Y - I_p}\fs -\norm{X\tp X - I_p}\fs \\
			={}& \inner{4D, X(X\tp X - I_p)} + 4\inner{\Phi(D\tp X), \Phi(D\tp X)} + 2\inner{D\tp D, X\tp X - I_p} + \ca{O}(\norm{D}\ff^3),
		\end{aligned}
	\end{equation*}
	the hessian of $h(X)$ can be expressed as 
	\begin{equation*}
		\nabla^2 h(X)[D] = \nabla^2 g(X) [D] + \beta\left[2 X\Phi(X\tp D) +  D (X\tp X - I_p)\right].
	\end{equation*} 
\end{proof}

Next, we give an important equality.
\begin{lem}
	\label{Le_inner_g_X}
	For any $X \in \bb{R}^{n\times p}$, we have
	\begin{equation*}
		\inner{X(X\tp X - I_p), \nabla g(X)} =  -\frac{3}{2}\inner{\left( X\tp X - I_p \right)^2, \Phi(X\tp G(X))}.
	\end{equation*}
\end{lem}
\begin{proof}
	Consider the inner product of $\nabla g(X)$ and $X(X\tp X - I_p)$, the following equality holds for any $X \in \bb{R}^{n\times p}$,
	\begin{equation*}
		\begin{aligned}
			&\inner{X(X\tp X - I_p), \nabla g(X)} \\
			={}& \inner{X(X\tp X - I_p), G(X)\Apen{X}} - \inner{X(X\tp X - I_p), X \Phi(X\tp G(X))}\\
			={}& \inner{(X\tp X - I_p)\Apen{X} , \Phi(X\tp G(X))} - \inner{(X\tp X - I_p)X\tp X,  \Phi(X\tp G(X))}\\
			={}& -\frac{3}{2}\inner{\left( X\tp X - I_p \right)^2, \Phi(X\tp G(X))}.
		\end{aligned}
	\end{equation*}
\end{proof}

Finally, we arrive at the main proposition in this preliminary section.
\begin{prop}
	\label{Prop_uniformly_bounded}
	Suppose  Assumption \ref{Assumption_global_f} holds, and $\tX$ is a first-order stationary point of \ref{Prob_Pen}, then $\norm{\tX}_2 \leq 1 + \frac{\hat{M}_1}{\beta}$. Furthermore, when $\beta \geq \hat{\beta}$, we can conclude that all the first-order stationary points of \ref{Prob_Pen} are contained in $\Omega$. 
\end{prop}
\begin{proof}
	Let $\tX = U\Sigma V\tp$ be the singular value decomposition of $\tX$,
	namely, $U\in \bb{R}^{n\times p}$ and $V\in \bb{R}^{n\times p}$ are the orthogonal matrices and $\Sigma$
	is a diagonal matrix with singular values of $\tX$ on its diagonal and $\sigma_1 \leq ...\leq \sigma_p$. 
	Suppose the statement to be proved is not true, we achieve $\sigma_p > 1 + \frac{\hat{M}_1}{\beta}$. 
	Let $\tilde{D} := U \Diag(0,...,0, 1) V\tp$, then from the first-order optimality condition, we have
	\begin{equation*}
		\inner{\nabla h(\tilde{X}), \tilde{D}} = 0.
	\end{equation*}
	Besides, Assumption \ref{Assumption_global_f} illustrates that $\norm{G(X)}\ff$ is bounded and thus
	\begin{equation*}
		\begin{aligned}
			&|\inner{\tilde{D},\nabla g(\tX)}| \\
			={}& \|\inner{\tilde{D}, G(X)\left( \frac{3}{2}I_p - \frac{1}{2} X\tp X \right) - \inner{D, X\Phi(X\tp G(X))}} \|\\
			\leq{}& \left| \tr\left( \left( \frac{3}{2}I_p - \frac{1}{2} X\tp X \right)D\tp G(X) \right) \right| + \frac{1}{2}\left| \tr\left( D\tp XX\tp G(X) \right) \right| + \frac{1}{2} \left|\tr\left(  G(X)\tp XD\tp X \right)  \right|\\
			\leq{}& \left|\frac{3}{2} - \frac{1}{2}\sigma_{p}^2\right| \hat{M}_1 + \frac{1}{2}\sigma_{p}^2 \hat{M}_1 + \frac{1}{2}\sigma_{p}^2 \hat{M}_1\\
			\leq{}& \frac{3(\sigma_{p}^2 + 1)}{2}\hat{M}_1. 
		\end{aligned}
	\end{equation*}
	On the other hand, from the definition of $\tilde{D}$, we can conclude that
	\begin{equation*}
		\inner{\tX(\tX\tp\tX - I_p), \tilde{D}} = \sigma_p(\sigma_p^2-1)  .
	\end{equation*}
	Notice that when $\beta \geq 3\hat{M}_1$, for any $t \geq 1 + \frac{\hat{M}_1}{\beta}$, it holds that 
	\begin{equation}
		\frac{(t^2-1)t}{t^2 + 1} > \frac{(t^2-1)}{t^2 + 1} \geq 1 - \frac{2}{t^2 +1} \geq 1 - \frac{1}{1 +  \frac{3\hat{M}_1}{\beta}} \geq \frac{3 \hat{M}_1}{2\beta}. 
	\end{equation}
	Therefore, when $\beta \geq 3\hat{M}_1$, we achieve
	\begin{equation*}
		\inner{\nabla h(\tX), \tilde{D}} \geq \beta \inner{\tX(\tX\tp\tX - I_p), \tilde{D}} - |\inner{\tilde{D},\nabla g(\tX)}| \geq (\sigma_p^2-1) (\beta \sigma_{p} ) - \frac{3(\sigma_{p}^2 + 1)}{2}\hat{M}_1 >0,
	\end{equation*}
	which contradicts to the first-order optimality. Therefore, we can conclude that $\norm{\tX}_2 \leq 1 + \frac{\hat{M}_1}{\beta}$. 
\end{proof}

	Additionally, in the following proposition, we illustrate that Assumption \ref{Assumption_global_f} implies that \ref{Prob_Pen} is bounded below for any $\beta > 0$. 
\begin{prop}
	Suppose Assumption \ref{Assumption_global_f} holds, then for any $\beta > 0$, \ref{Prob_Pen} is bounded below over $\bb{R}^{n\times p}$. 
\end{prop}
\begin{proof}
	For any $X \in \bb{R}^{n\times p}$, it holds from Assumption \ref{Assumption_global_f} that 
	\begin{equation*}
		\begin{aligned}
			&|g(X) - g(0)| = \left| f\left(X\Apen{X}\right) - f(0) \right| \leq \hat{M}_1 \norm{X\Apen{X}}\ff\\
			\leq{}& p\hat{M}_1 \left( \frac{1}{2} \norm{X}_2^3 + \frac{3}{2} \norm{X}_2 \right). 
		\end{aligned}
	\end{equation*}
	Moreover, it holds that $\norm{X\tp X - I_p}\ff\geq (\norm{X}_2^2 -1)^2$. Therefore, for any $X \in \bb{R}^{n\times p}$ that satisfies $\norm{X}_2 \geq \frac{128p\hat{M}_1}{\beta} + 2$, it holds that $(\norm{X}_2^2 -1)^2 \geq \frac{1}{16} \norm{X}_2^4$. Then we achieve 
	\begin{equation*}
		\begin{aligned}
			&h(X) - h(0) \geq \frac{\beta}{4} \norm{X\tp X - I_p}\ff - |g(X) - g(0)|\\
			\geq{}& \frac{\beta}{4}(\norm{X}_2^2 -1)^2 -  p\hat{M}_1 \left( \frac{1}{2} \norm{X}_2^3 + \frac{3}{2} \norm{X}_2 \right)\\
			\geq{}& \frac{\beta}{64}\norm{X}_2^4 - p\hat{M}_1\left( \frac{1}{2} \norm{X}_2^3 + \frac{3}{2} \norm{X}_2 \right) \geq \frac{\beta}{64} \norm{X}_2^4 - 2p\hat{M}_1\norm{X}_2^3 > 0. 
		\end{aligned}
	\end{equation*}
	As a result, it holds that 
	\begin{equation*}
		\begin{aligned}
			&\inf_{X \in \bb{R}^{n\times p}} h(X) = \min\left\{\inf_{\norm{X}_2 \leq \frac{128p\hat{M}_1}{\beta} + 2} h(X), \inf_{\norm{X}_2 \geq \frac{128p\hat{M}_1}{\beta} + 2} h(X)  \right\}\\
			\geq{}& \min\left\{\inf_{\norm{X}_2 \leq \frac{128p\hat{M}_1}{\beta} + 2} h(X), h(0)  \right\}
			> -\infty. 
		\end{aligned}
	\end{equation*}
	Hence we complete the proof. 
\end{proof}


\subsection{Computational complexity of the first-order oracle}

In this subsection, we analyze the cost
of
calculating the first-order derivative of $h(X)$,
which takes the main computational cost in each iterate
of a first-order algorithm such as gradient descent methods, nonlinear conjugate gradient methods, etc.
Then we compare it with the fundamental operations in Riemannian optimization approaches.  From the expression for $\nabla h(X)$ illustrated in Lemma \ref{Le_gradient_h}, we find  that computing $\nabla h(X)$ only involves computing $\nabla f$ and matrix-matrix multiplication. The computational cost of the basic linear algebra operations and the overall costs of computing the gradient of $h$ are listed in Table \ref{Table_gradient_h}, while a comparison between several fundamental operations in Riemannian optimization and their corresponding operations for $h(X)$ are listed in Table \ref{Table_Compare}. 
Here, $\rm{FO}$ denotes the computational costs of computing the gradient of $f$, and those terms in bold stand for the operations that cannot be parallelized.

\begin{table}[htbp]
	\small
	\centering
	\begin{tabular}{|c|c|c|}
		\hline
		\multirow{4}{*}{Compute $\nabla f\left( X \Apen{X} \right) \Apen{X}$}& $X\tp X$ & $np^2$ \\ 
		& $X(X\tp X - I_p)$ & $2np^2$ \\ 
		& $G(X) = \nabla f\left(Y \right)\left\vert_{Y =  X \Apen{X}} \right.$ & $1\rm{FO}$ \\ 
		& $G(X) \Apen{X}$ & $2np^2$ \\ \hline
		\multirow{2}{*}{Compute $X \Phi\left(X\tp G(X)\right)$} & $\Phi\left(X\tp G(X)\right)$ & $2np^2$ \\
		& $X \Phi\left(X\tp G(X)\right)$ & $2np^2$ \\\hline
		In total & \multicolumn{2}{c|}{$1\rm{FO} + 9np^2$} \\ \hline
	\end{tabular}
	\caption{Computational complexity the first-order oracle in \Expen. }
	\label{Table_gradient_h}
\end{table}


{ 
\begin{table}[htbp]
	\scriptsize
	\centering
	\begin{tabular}{|cc|cc|cc|}
		\hline
		\multicolumn{2}{|c|}{Riemannian optimization approaches} & \multicolumn{2}{c|}{\ref{Prob_Pen} based approaches} & \multicolumn{2}{|c|}{PLAM \cite{gao2019parallelizable}}\\ \hline
		\multirow{2}{*}{Riemannian gradient} & $\nabla f(X) - X \nabla f(X)\tp X$ &   \multirow{2}{*}{Euclidean gradient}  &  $\nabla h(X)$     & \multirow{2}{*}{Descending direction}  &  $\nabla_x \ca{L}_{\beta}(X, \Phi(\nabla f(X)\tp X))$      \\
		& $1\rm{FO} + 4np^2$ \cite{gao2019parallelizable}  &             & $1\rm{FO} + 9np^2 $ &             & $1\rm{FO} + 7np^2 $ \cite{gao2019parallelizable} \\ \hline
		\multirow{2}{*}{Retraction}&  Cholesky factorization: $3np^2 + {\bf \ca{O}(p^3)}$  &  \multirow{2}{*}{No retraction}           &       ---    &\multirow{2}{*}{No retraction}           &       ---       \\
		& Gram-Schmidt: ${\bf 2np^2}$           &             &     ---  &             &     ---       \\ \hline
		\multirow{2}{*}{Vector transport \cite{Absil2009optimization}} & $\xi_X \in \ca{T}_X \to \xi_X - Y \Phi(Y\tp \xi_X) \in \ca{T}_Y$   &   \multirow{2}{*}{No vector transport}  &  ---  &   \multirow{2}{*}{No vector transport}  &  ---         \\     
		& $4np^2$           &             & --- &             & --- \\ \hline
	\end{tabular}
	\caption{Comparison on the computational complexity 
		of the first-order oracles among Riemannian optimization
		approaches, \Expens based approaches and the specialized optimization algorithm PLAM. Here $\ca{L}_{\beta}(X, \Lambda):= f(X) - \frac{1}{2} \inner{X\tp X - I_p, \Lambda} + \frac{\beta}{4} \norm{X\tp X - I_p}\fs$.} 
	\label{Table_Compare}
\end{table}
}

\section{Properties of \ref{Prob_Pen}}

In this  section, we analyze the theoretical properties of \ref{Prob_Pen}. 

\subsection{First-order relationship}

In this subsection, we study the first-order relationship between OCP and \Expen.  The main theoretical results of this subsection can be summarized in Figure \ref{Fig_roadmap_FOSP}. 
Here ``A.", ``D.", ``P.", ``T." are the abbreviations of  ``Assumption", ``Definition", ``Proposition", and ``Theorem", respectively.

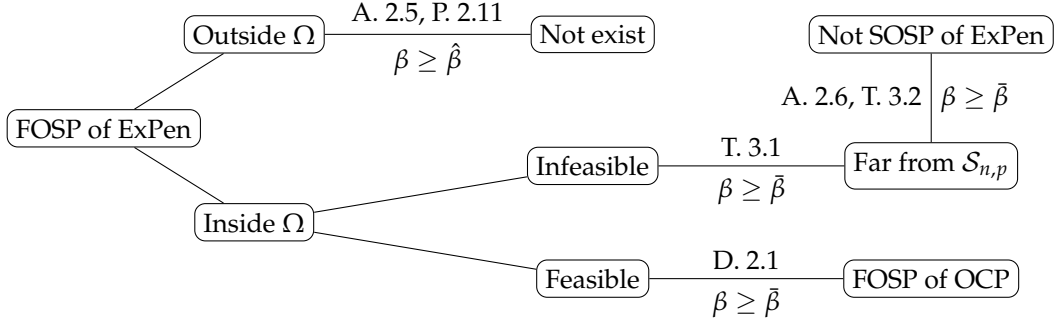
\begin{figure}[htbp]
	\tikzstyle{level 1}=[level distance=2cm, sibling distance=2.5cm]
	\tikzstyle{level 2}=[level distance=4.5cm, sibling distance=1.5cm]
	\tikzstyle{level 3}=[level distance=4.5cm, sibling distance=1.5cm]
	
	\tikzstyle{smallbag} = [text width=8em, shape=rectangle, rounded corners,draw, align=center]
	\tikzstyle{bag} = [ shape=rectangle, rounded corners,draw, align=center]
	\tikzstyle{end} = [circle, minimum width=3pt,fill, inner sep=0pt]
	\begin{tikzpicture}[grow=right]
		\node[bag] { FOSP of \ref{Prob_Pen}}
		child {
			node[bag] {Inside $\Omega$}        
			child {
				node[bag] {Feasible}
				child {
					node[bag] {FOSP of \ref{Prob_Ori}}
					edge from parent
					node[above] {D. \ref{Defin_FOSP}}
					node[below] {$\beta \geq \bar{\beta}$}
				}
				edge from parent
			}
			child {
				node[bag] {Infeasible}
				child {
					node[bag] {Far from $\ca{S}_{n,p}$}
					child[grow = up,level distance=1.75cm] {
						node[bag] {Not SOSP of \ref{Prob_Pen}}
						edge from parent
						node[left] {A. \ref{Assumption_second_order}, T. \ref{The_strict_saddle}}
						node[right] {$\beta \geq \bar{\beta}$}
					}
					edge from parent
					node[above] {T. \ref{The_Equivalence_local}}
					node[below] {$\beta \geq \bar{\beta}$}
				}
				edge from parent
			}
			edge from parent 
		}
		child {
			node[bag] {Outside $\Omega$}        
			child {
				node[bag] {Not exist}
				edge from parent
				node[above] {A. \ref{Assumption_global_f}, P. \ref{Prop_uniformly_bounded}}
				node[below] {$\beta \geq \hat{\beta}$}
			}
			edge from parent         
		};
	\end{tikzpicture}
	\caption{Roadmap of the first-order relationship between \ref{Prob_Ori} and \ref{Prob_Pen}.
	}
	\label{Fig_roadmap_FOSP}
\end{figure}

The following theorem categorizes the first-order stationary points of \ref{Prob_Pen} in $\Omega$. 
\begin{theo}
	\label{The_Equivalence_local}
	Suppose $X^* \in \Omega$ is a first-order stationary point of \ref{Prob_Pen},  and $\beta \geq \bar{\beta}$, then either $X^*$ is a first-order stationary point of \ref{Prob_Ori}, or $\sigma_{\min}(X^*) \leq \sqrt{\frac{2M_1}{\beta}}$. 
\end{theo}
\begin{proof}
	Suppose $\sigma_{\min}(X^*) >  \sqrt{\frac{2M_1}{\beta}}$, then $\beta {X^*}\tp {X^*} - 2M_1 I_p$ is positive definite. 
	Besides, from Lemma \ref{Le_inner_g_X} we achieve
	\begin{equation*}
		\begin{aligned}
			0 ={}& \inner{\nabla h({X^*}), {X^*}({X^*}\tp {X^*} - I_p)}\\
			\geq {}&   \inner{\beta {X^*}({X^*}\tp {X^*} - I_p),{X^*}({X^*}\tp {X^*} - I_p) } - \left| \inner{\nabla g({X^*}), {X^*}({X^*}\tp {X^*} - I_p)} \right| \\
			\geq{}& \inner{\beta {X^*}({X^*}\tp {X^*} - I_p),{X^*}({X^*}\tp {X^*} - I_p) } - \frac{3}{2} \norm{{X^*}\tp G(X^*)}_2 \tr\left( \left({X^*}\tp X^* - I_p\right)^2 \right)  \\
			\geq{}& \inner{\beta {X^*}({X^*}\tp {X^*} - I_p),{X^*}({X^*}\tp {X^*} - I_p) } - \inner{\frac{3}{2} \norm{X^*}_2\norm{G(X^*)}\ff \cdot I_p, ({X^*}\tp {X^*} - I_p)^2} \\
			\geq{}& \inner{\beta {X^*}\tp {X^*} - 2M_1 I_p, ({X^*}\tp {X^*} - I_p)^2}  \geq 0,
		\end{aligned}
	\end{equation*}
	which illustrates that ${X^*}\tp X^* = I_p$. Then we can conclude that $0 = \nabla h(X^*) = \grad f(X^*)$ and thus complete the proof. 
\end{proof}

As illustrated in Theorem \ref{The_Equivalence_local}, any first-order stationary point of \ref{Prob_Pen} in $\Omega$ is either a first-order stationary point of \eqref{Prob_Ori}, or is  far from the Stiefel manifold. The following theorem illustrates that any infeasible first-order stationary point of \ref{Prob_Pen} cannot be a second-order stationary point of $h(X)$.

\begin{theo}
	\label{The_strict_saddle}
	Suppose Assumption \ref{Assumption_second_order} holds, $\beta \geq \bar{\beta}$, then any infeasible first-order stationary point $\tX$ of \ref{Prob_Pen} in $\Omega$ is not a second-order stationary point of \ref{Prob_Pen}. More specifically, $\lambda_{\min}(\nabla^2 h(\tX)) \leq -\frac{\beta}{24}$. 
\end{theo}
\begin{proof}
	Suppose the statement is not true, namely, $\tX$ is a second-order stationary point of \ref{Prob_Ori}. 
	Since $\beta \geq 12M_1$, it holds that 
	$\sigma_{\min}(\tX\tp \tX) \leq \frac{1}{6}$ by Proposition \ref{The_Equivalence_local}.
	Let $\tX = U\Sigma V\tp$ be the singular value decomposition of $\tX$,
	namely, $U\in \bb{R}^{n\times p}$ and $V\in \bb{R}^{n\times p}$ are the orthogonal matrices and $\Sigma$
	is a diagonal matrix with singular values of $\tX$ on its diagonal. 
	Without loss of generality, we assume $\sigma_1\leq \frac{1}{\sqrt{6}}$ 
	which is the first entry 
	of the diagonal matrix $\Sigma$. 
	
	Then we denote 
	$D=-u_1v_1\tp $, where 
	$u_1$ and $v_1$ are the first columns of $U$ and $V$, respectively.
	It holds that 
	\begin{equation*}
		(\tX + tD)\tp(\tX+ tD) =  \tX\tp \tX  + 2tD\tp \tX+ t^2D\tp D 
		= V\tp \Sigma^2 V -2t\sigma_1 v_1v_1\tp + t^2 v_1v_1\tp.
	\end{equation*}
	Due to the first-order stationarity of $\tX$, it holds $\nabla h(\tX) = 0$
	which implies $D\tp \nabla h(\tX) = 0$. 
	First, we have
	\begin{equation*}
		\begin{aligned}
			&\norm{(\tX + tD)\tp(\tX+ tD) - I_p}\fs = \norm{\tX\tp \tX - I_p  + 2t\Phi(\tX\tp D) + t^2D\tp D }\fs \\
			\geq{}& \norm{\tX\tp\tX - I_p}\fs + 4t^2 \inner{\Phi(\tX\tp D), \Phi(\tX\tp D)} + 2t^2 \inner{D\tp D, \tX\tp\tX - I_p} + \ca{O}(t^3)\\
			\geq{}& \norm{\tX\tp\tX - I_p}\fs + 4t^2 \sigma_1^2 - 2t^2(1-\sigma_1^2) + \ca{O}(t^3) 
			\geq{} \norm{\tX\tp\tX - I_p}\fs - t^2 + \ca{O}(t^3).
		\end{aligned}
	\end{equation*}
	As a result, 
	\begin{equation*}
		\begin{aligned}
			& h(\tX + tD) 
			\leq{} h(\tX) + t\cdot \inner{ D, \nabla h(\tX) }+ \frac{t^2}{2} \norm{\nabla^2 g(\tX)}\ff\norm{D}\fs
			- \frac{\beta}{8}t^2  + \ca{O}(t^3)\\
			\leq{} & h(\tX) + \frac{t^2}{2} M_2 - \frac{\beta t^2}{8} + \ca{O}(t^3)
			\leq{}  h(\tX)
			-\frac{t^2 \beta }{24} +\ca{O}(t^3),
		\end{aligned}
	\end{equation*}
	which contradicts to the second-order optimality of \ref{Prob_Pen}. Therefore, we can conclude that any infeasible first-order stationary point of \ref{Prob_Pen} in $\Omega$ is not a second-order stationary point of \ref{Prob_Pen}, and $\lambda_{\min}(\nabla^2 h(\tX)) \leq -\frac{\beta}{24}$.
\end{proof}

Combining Proposition \ref{Prop_uniformly_bounded} with Theorem \ref{The_Equivalence_local}, we arrive at the following
corollary. 
\begin{coro}
	\label{Coro_Equivalence}
	Suppose Assumptions \ref{Assumption_global_f} and \ref{Assumption_second_order} hold, and $\beta \geq \hat{\beta}$. Let $X^*$ be a first-order stationary point of \ref{Prob_Pen}, then either $X^*$ is a first-order stationary point of \ref{Prob_Ori}, or is far from the Stiefel manifold and can not be a second-order stationary point of \ref{Prob_Pen}.
\end{coro}
The proof of this corollary is straightforward and hence omitted.

\subsection{Second-order relationship}
In this subsection,  we study the first-order relationship between OCP and \Expen.  
The main theoretical results of this subsection can be summarized in Figure \ref{Fig_roadmap_SOSP}. 

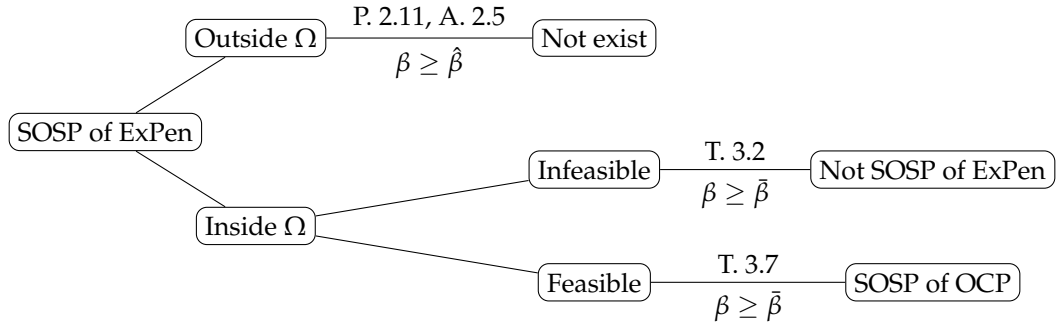
\begin{figure}[h!]
	\tikzstyle{level 1}=[level distance=2cm, sibling distance=2.5cm]
	\tikzstyle{level 2}=[level distance=4.5cm, sibling distance=1.5cm]
	\tikzstyle{level 3}=[level distance=4.5cm, sibling distance=1.5cm]
	
	\tikzstyle{smallbag} = [text width=8em, shape=rectangle, rounded corners,draw, align=center]
	\tikzstyle{bag} = [ shape=rectangle, rounded corners,draw, align=center]
	\tikzstyle{end} = [circle, minimum width=3pt,fill, inner sep=0pt]
	\begin{tikzpicture}[grow=right]
		\node[bag] { SOSP of \ref{Prob_Pen}}
		child {
			node[bag] {Inside $\Omega$}        
			child {
				node[bag] {Feasible}
				child {
					node[bag] {SOSP of \ref{Prob_Ori}}
					edge from parent
					node[above] {T. \ref{The_Equivalence_local_second_order}}
					node[below] {$\beta \geq \bar{\beta}$}
				}
				edge from parent
			}
			child {
				node[bag] {Infeasible}
				child {
					node[bag] {Not SOSP of \ref{Prob_Pen}}
					edge from parent
					node[above] {T. \ref{The_strict_saddle}}
					node[below] {$\beta \geq \bar{\beta}$}
				}
				edge from parent
			}
			edge from parent 
		}
		child {
			node[bag] {Outside $\Omega$}        
			child {
				node[bag] {Not exist}
				edge from parent
				node[above] {P. \ref{Prop_uniformly_bounded}, A. \ref{Assumption_global_f}}
				node[below] {$\beta \geq \hat{\beta}$}
			}
			edge from parent         
		};
	\end{tikzpicture}
	\caption{Roadmap of the second-order relationship between \ref{Prob_Ori} and \ref{Prob_Pen} under Assumption \ref{Assumption_second_order}. 
	}
	\label{Fig_roadmap_SOSP}
\end{figure}

We first analyze the relationship between Riemannian hessian of the original objective function $f$ and the Euclidean hessian of the penalty function $h$. 

\begin{lem}
	\label{Lem_Hessian_tangent}
	Suppose $f(X)$ satisfies Assumption \ref{Assumption_second_order}. Then for any given $X \in \ca{S}_{n,p}$ and any $D_1 \in \ca{T}_X$, the following equality holds,
	\begin{equation}
		\inner{D_1, \nabla^2 h(X)[D_1]}  =  \inner{D_1, \nabla^2 f(X)[D_1] - D_1 \Phi(X\tp \nabla f(X))}.
	\end{equation}
\end{lem}
\begin{proof}
	Since $D_1 \in \ca{T}_X$, by the definition of $\ca{T}_X$ we have $\Phi(D_1\tp X) = 0$. 
	Moreover, the definition of $\PT$ indicates that $\PT(D_1) = D_1$.
	As a result,  from Proposition \ref{Prop_Hessian} we can conclude that
	\begin{equation}
		\label{Eq_Le_tangent_hessian}
		\begin{aligned}
			&\nabla^2 h(X) [D_1] = \nabla^2 g(X) [D_1] \\
			={}& \PT  \left( \ca{H}(X)[\PT(D_1)]\right)- D_1 \Phi(X\tp \nabla f(X)) - \nabla f(X) \Phi(X\tp D_1) - X \Phi(D_1\tp \nabla f(X))\\
			={}& \PT\left(\nabla^2 f(X)[D_1]\right) - D_1 \Phi(X\tp \nabla f(X)) - X\Phi(D_1\tp \nabla f(X)).
		\end{aligned}
	\end{equation}
	Therefore, for any $D_1 \in \ca{T}_X$, we have
	\begin{equation*}
		\begin{aligned}
			&\inner{D_1, \nabla^2 h(X)[D_1]} =  \inner{D_1, \nabla^2 f(X)[D_1] - D_1 \Phi(X\tp \nabla f(X))} - \inner{\Phi(D_1\tp X), \Phi(D_1\tp \nabla f(X)) } \\
			={}& \inner{D_1, \nabla^2 f(X)[D_1] - D_1 \Phi(X\tp \nabla f(X))} .
		\end{aligned}
	\end{equation*}
\end{proof}

\begin{lem}
	\label{Le_tangent_hessian}
	Suppose $f(X)$ satisfies Assumption \ref{Assumption_second_order} and $X \in \ca{S}_{n,p}$ is a first-order stationary point of $h(X)$. Then for any $D_1 \in \ca{T}_X$ and any $D_2 \in \ca{N}_X$, we have 
	\begin{align}
		& \inner{D_2, \nabla^2 h(X)[D_2]} \geq (2\beta - M_2)\norm{D_2}\fs; \label{Eq_normal_hessian}\\
		& \inner{D_1,\nabla^2 h(X)[D_2] } = 0. \label{Eq_cross_hessian}
	\end{align}
\end{lem}
\begin{proof}
	Since $D_1 \in \ca{T}_X$, by the definition of $\ca{T}_X$ we have $\Phi(D_1\tp X) = 0$. 
	Besides, the definition of $\PT$ indicates that $\PT(D_1) = D_1$.  Since $D_1 \in \ca{T}_X$ and $D_2 \in \ca{N}_X$, we have $\PT(D_2) = 0$. Besides, there exists a symmetric matrix $\Lambda_2 \in \bb{R}^{p\times p}$  such that $D_2 = X\Lambda_2$. Then we have
	\begin{equation*}
		\begin{aligned}
			&\inner{D_2,\nabla^2 h(X)[D_1] } = \inner{D_2, \PT(\nabla^2 f(X)[\PT(D_1)]) - D_1 \Phi(X\tp \nabla f(X)) - X\Phi(D_1\tp \nabla f(X))}\\
			={}& - \inner{D_2, D_1 \Phi(X\tp \nabla f(X))} - \inner{D_2,  X\Phi(D_1\tp \nabla f(X))}\\
			={}& -\tr\left( \Lambda_2 X\tp D_1 X\tp \nabla f(X) \right) - \tr\left( \Lambda_2 \Phi(D_1\tp X X\tp \nabla f(X)) \right)\\
			\overset{(i)}{=}& \tr\left( \Lambda_2  D_1\tp X X\tp \nabla f(X) \right) - \tr\left( \Lambda_2 \Phi(D_1\tp X X\tp \nabla f(X)) \right)\\
			\overset{(ii)}{=}& \tr\left( \Lambda_2 \left[  D_1\tp X X\tp \nabla f(X) - \Phi( D_1\tp X X\tp \nabla f(X) ) \right] \right)
			={} 0.
		\end{aligned}
	\end{equation*}
	Here $(i)$ follows the fact that $D_1\tp X$ is skew-symmetric, and $(ii)$ directly uses the fact that for any $T \in \bb{R}^{p\times p}$, $T - \Phi(T) $ is skew-symmetric.

	Moreover, notice that  $D_2 \in \ca{N}_X$  implies that $\norm{\Phi(D_2\tp X)}\ff = \norm{D}\ff$. Then by the expression of $\nabla^2 h(X)$ in Proposition \ref{Prop_Hessian}, we can conclude that 
	\begin{equation*}
		\inner{D_2, \nabla^2 h(X)[D_2]} = \inner{D_2, \nabla^2 g(X)[D_2]} + 2\beta \norm{\Phi(D_2\tp X)}\fs \geq (2\beta - M_2)\norm{D_2}\fs,
	\end{equation*}
	which completes the proof. 
\end{proof}

\begin{theo}
	\label{The_Eig_Correspondence}
	Suppose $f(X)$ satisfies Assumption \ref{Assumption_second_order}, for any first-order stationary point $X \in \ca{S}_{n,p}$ of \ref{Prob_Ori}, any eigenvalue of $\hess f(X)$ is an eigenvalue of $\nabla^2 h(X)$. In turn, any eigenvalue of $\nabla^2 h(X)$ is either an eigenvalue of  $\hess f(X)$, or greater than $2\beta - M_2$. 
\end{theo}
\begin{proof}
	Let $\sigma_1$ be an eigenvalue of $\hess f(X)$, it follows from Definition \ref{Defin_SOSP} that there exists an $\hat{D}_1 \in \ca{T}_X$ such that 
	\begin{equation*}
		\PT\left(\nabla^2 f(X) [\hat{D}_1] - \hat{D}_1 \Phi(X\tp \nabla f(X))\right) = \sigma D_1.  
	\end{equation*} 
	In addition, Lemma \ref{Lem_Hessian_tangent} indicates that
	\begin{equation*}
		\nabla^2 h(X)[\hat{D}_1] = \PT(\nabla^2 f(X) [\hat{D}_1] - \hat{D}_1 \Phi(X\tp \nabla f(X))) = \sigma D_1.
	\end{equation*}
	Therefore, any eigenvalue of $\hess f(X)$ is an eigenvalue of $\nabla^2 h(X)$. 
	
	On the other hand, Lemma \ref{Lem_Hessian_tangent} and Lemma \ref{Le_tangent_hessian} verify that the linear operator $\nabla ^2 h(X)$ maps a vector in $\ca{T}_X$ or $\ca{N}_X$ to $\ca{T}_X$ or $\ca{N}_X$, respectively. Then any eigenvector of $\nabla^2 h(X)$ is  either in $\ca{T}_X$ or $\ca{N}_X$. For any $D_2 \in \ca{N}_X$, \eqref{Eq_normal_hessian} implies that 
	\begin{equation*}
		\inner{D_2, \nabla^2 h(X)[D_2]} \geq (2\beta - M_2) \norm{D_2}\fs,
	\end{equation*}
	from which we can conclude that any eigenvalue of $\nabla^2 h(X)$ is either an eigenvalue of  $\hess f(X)$, or greater than $2\beta - M_2$.

\end{proof}

Based on Theorem \ref{The_Eig_Correspondence}, we can establish the second-order relationship between OCP and \Expen.

\begin{theo}
	\label{The_Equivalence_local_second_order}
	Suppose $f(X)$ satisfies Assumption \ref{Assumption_second_order} and $\beta \geq \bar{\beta} $, then any second-order stationary point of $h(X)$ in $\Omega$ is a second-order stationary point of \ref{Prob_Ori}. Moreover,  \ref{Prob_Ori} and \ref{Prob_Pen} have exactly the same second-order stationary points in $\Omega$.
\end{theo}
\begin{proof}
	Let $X \in \Omega$ be a second-order stationary point of $h(X)$, then all eigenvalues of $\nabla^2 h(X)$ are nonnegative. It follows from Theorems \ref{The_Equivalence_local} and \ref{The_strict_saddle} that $X$ is feasible.  We can further conclude  $X$ is a first-order stationary point of \ref{Prob_Ori} by
	Proposition \ref{Le_gradient_h}. In addition,  Theorem \ref{The_Eig_Correspondence} shows that all eigenvalues of $\hess f(X)$ consist of a subset of the spectra of $\nabla^2 h(X)$, and thus are nonnegative. Namely, $X$ is a second-order stationary point of \ref{Prob_Ori}. 
	
	In turn, let $X \in \ca{S}_{n,p}$ be a second-order stationary point of \ref{Prob_Ori},  naturally, all the eigenvalues of $\hess f(X)$ are nonnegative.
	Then we can immediately obtain that all the eigenvalues of 
	$\nabla^2 h(X)$ are nonnegative
	resulting from Theorem \ref{The_Eig_Correspondence} and the fact that $2\beta\geq M$.  Hence, $X$ is a second-order stationary point of $h(X)$. 
	
\end{proof}

Based on the second-order relationship between OCP and \Expens in $\Omega$ illustrated in Theorem \ref{The_Equivalence_local_second_order}, we can immediately obtain their second-order relationship
in $\bb{R}^{n\times p}$ by utilizing Proposition \ref{Prop_uniformly_bounded}. We omit the proof, since it is quite straightforward.

\begin{coro}
	\label{The_Equivalence_global_second_order}
	Suppose Assumptions \ref{Assumption_global_f} and \ref{Assumption_second_order} hold, and $\beta \geq  \hat{\beta}$, then \ref{Prob_Ori} and \ref{Prob_Pen} share
	the same second-order stationary points.
\end{coro}

\subsection{Estimating Stationarity}
When we implement an infeasible approach to \ref{Prob_Pen}, the returned solution is usually infeasible since the iterates are not necessarily restricted on $\ca{S}_{n,p}$. 
Sometimes we pursue high accuracy for the feasibility at the same time. To this end, we impose an orthonormalization as a postprocess after obtaining a solution $X$ with mild accuracy by applying an unconstrained optimization approach to solve \ref{Prob_Pen}. Namely, 
\begin{eqnarray}\label{orth}
	X \to \ProjS(X),
\end{eqnarray}
where $\ProjS: \bb{R}^{n\times p} \to \ca{S}_{n,p}$ is the projection on Stiefel manifold defined in Section \ref{Sec:notation}. 

In this subsection, we study the relationship between the stationarity of $\ProjS(X)$ with respect to \ref{Prob_Ori} and the stationarity $X$ with respect to \ref{Prob_Pen}. More specifically, we aim to estimate an upper-bound for $\norm{\grad f(\ProjS(X))}\ff$ from $\norm{\nabla h(X)}\ff$. Therefore, we could explicitly setting the stopping criteria for \ref{Prob_Pen} to achieve a desired accuracy in solving \ref{Prob_Ori}. Moreover, the iteration complexity of various unconstrained optimization approaches directly follows from existing rich results when applied to solve \ref{Prob_Pen}.

The following lemma guarantees that the postprocess \eqref{orth}  can further reduce the function value 
if the current iterate is sufficiently close to the Stiefel manifold.
\begin{prop}
	\label{Le_ortho_fval}
	
	Suppose 
	$X \in\BOMs$, then it holds that
	\begin{equation*}
		h(\ProjS(X)) \leq h(X) - \left( \frac{\beta}{4} -  \frac{M_1}{2} \right)\norm{X\tp X - I_p}\fs. 
	\end{equation*}
\end{prop}	
\begin{proof}
	By the SVD of $X$, we first conclude that 
	\begin{equation}
		\label{Eq_Le_ortho_fval_0}
		\begin{aligned}
			&\norm{X \left( \frac{3}{2}I_p - \frac{1}{2} X\tp X \right) - \ProjS(X)}\ff = \norm{U \Sigma\left(\frac{3}{2}I_p - \frac{1}{2}\Sigma^2 \right) V\tp - UV\tp }\ff \\
			={}& \norm{\Sigma\left(\frac{3}{2}I_p - \frac{1}{2}\Sigma^2 \right) - I_p}\ff =\norm{\left(\frac{1}{2} \Sigma + I_p\right)\left( \Sigma -  I_p\right)^2}\ff \\
			\overset{(i)}{\leq}{}& \frac{1}{2} \norm{\Sigma^2 - I_p}\fs = \frac{1}{2}\norm{X\tp X - I_p}\fs.
		\end{aligned}
	\end{equation}
	Here $(i)$ directly follows from $\left(\Sigma + 2I_p\right)\left( \Sigma -  I_p\right)^2 \preceq \left( \Sigma +  I_p\right)^2\left( \Sigma -  I_p\right)^2 $ when $\Sigma \succeq \frac{5}{6} I_p$,  and the fact that $\norm{A^2}\ff \leq \norm{A}\fs$ holds for any symmetric matrix $A$.

	Then we can conclude that
	\begin{equation*}
		\begin{aligned}
			&h(X) - h(\ProjS(X)) = f\left(X\Apen{X}\right) - f(\ProjS(X)) +  \frac{\beta}{4}\norm{X\tp X - I_p}\fs \\
			\overset{(ii)}{\geq}{}& -M_1 \norm{X\Apen{X} - \ProjS(X)}\ff + \frac{\beta}{4}\norm{X\tp X - I_p}\fs \\
			\overset{(iii)}{\geq}{}& \left( \frac{\beta}{4} -  \frac{M_1}{2} \right) \norm{X\tp X - I_p}\fs.
		\end{aligned}
	\end{equation*}
	Here $(ii)$ follows the Lipschitz continuity of $f$, and $(iii)$ is directly from \eqref{Eq_Le_ortho_fval_0}.
\end{proof}

\begin{lem}
	\label{Le_esti_h_g}
	For any $X\in\BOMs$, we have
	\begin{equation*}
		\norm{\nabla h(X)}\fs \geq \norm{\nabla g(X)}\fs + \left(\frac{2}{3}\beta^2- 4\beta M_1 \right)\norm{X\tp X - I_p}\fs.
	\end{equation*}
\end{lem}
\begin{proof}
	Since $h(X) = g(X) + \frac{\beta}{4}\norm{X\tp X - I_p}\fs$, we have
	\begin{equation*}
		\begin{aligned}
			&\inner{\nabla h(X), \nabla h(X)} = \inner{\nabla g(X), \nabla g(X)} + 2\beta \inner{\nabla g(X), X(X\tp X - I_p) } + \beta^2 \norm{X(X\tp X - I_p)}\fs \\
			=& \norm{\nabla g(X)}\fs - 3\beta \inner{\Phi(X\tp G(X)), (X\tp X - I_p)^2} + \beta^2 \norm{X(X\tp X - I_p)}\fs \\
			\geq{}&  \norm{\nabla g(X)}\fs + \left(\frac{2}{3}\beta^2- 4\beta M_1 \right)\norm{X\tp X - I_p}\fs. 
		\end{aligned}
	\end{equation*}
	Here the second equality directly follows Lemma \ref{Le_inner_g_X}. 
\end{proof}

The following lemma illustrates the relationship between $\norm{\nabla h(X)}\ff$ and $\norm{\grad f(\ProjS(X))}\ff$. 
\begin{lem}
	\label{Le_grad_f_upper_bound}
	Suppose $\beta \geq \bar{\beta}$, 
	$X \in \BOMs$, then it holds that
	\begin{equation*}
		\norm{\nabla h(X)}\ff \geq \frac{1}{2}\norm{\grad f(\ProjS(X))}\ff + \frac{\beta}{4} \norm{X\tp X - I_p}\ff. 
	\end{equation*}
\end{lem}
\begin{proof}
	Suppose $X$ has singular value decomposition as $X = U\Sigma V\tp$, we can conclude that 
	\begin{equation}
		\label{Eq_Le_grad_f_upper_bound}
		\norm{X  - \ProjS(X)}\ff = \norm{\Sigma - I_p}\ff \leq \frac{6}{11} \norm{\Sigma^2 - I_p}\ff =   \frac{6}{11}\norm{X\tp X - I_p}\ff.
	\end{equation}

	Therefore, the results in Lemma \ref{Le_esti_h_g} illustrates that
	\begin{equation*}
		\begin{aligned}
			&\norm{\nabla h(X)}\ff \geq \frac{1}{2}\norm{\nabla g(X)}\ff + \frac{\sqrt{6\beta^2 - 36\beta M_1}}{3}\norm{X\tp X - I_p}\ff \\
			\geq{}& \frac{1}{2}\norm{\nabla g(\ProjS(X))}\ff - \frac{3M_2}{11}\norm{X\tp X - I_p}\ff +  \frac{\sqrt{6\beta^2 - 36\beta M_1}}{3}\norm{X\tp X - I_p}\ff \\
			\geq{}& \frac{1}{2}\norm{\grad f(\ProjS(X))}\ff + \frac{\beta}{4} \norm{X\tp X - I_p}\ff. 
		\end{aligned}
	\end{equation*}
\end{proof}

\begin{prop}
	\label{Prop_hessian_esti}
	Suppose Assumption \ref{Assumption_second_order} holds, given any $X \in \BOMs$, then it holds that
	\begin{equation}
		\sigma_{\min}(\hess f(\ProjS(X)))  \geq \sigma_{\min}(\nabla^2 h(X)) -  \norm{\nabla^2 g(X) - \nabla^2 g(\ProjS(X))}\ff -\frac{9}{2}\norm{\nabla h(X)}\ff
	\end{equation}
\end{prop}
\begin{proof}
	Let $Y := \ProjS(X)$, i.e. let $X = U \Sigma V\tp$ be the singular value decomposition of $X$, then $Y = UV\tp$. Then from Lemma \ref{Lem_Hessian_tangent}, for any $D_1 \in \ca{T}_Y$, it holds that
	\begin{equation*}
		\begin{aligned}
			\inner{D_1, \nabla^2 h(Y)[D_1]} = \inner{D_1, \nabla^2 f(X)[D_1] - D_1 \Phi(X\tp \nabla f(X))}.
		\end{aligned}
	\end{equation*} 
	Therefore, the $\sigma_{\min}(\hess f(Y))$ can be expressed by
	\begin{equation*}
		\sigma_{\min}(\hess f(Y)) = \min_{D_1 \in \ca{T}_Y, \norm{D_1}\ff = 1} \inner{D_1, \nabla g(Y)[D_1]}. 
	\end{equation*}
	Let $\tilde{D}:= \arg\min_{D_1 \in \ca{T}_Y, \norm{D_1}\ff = 1} \inner{D_1, \nabla g(Y)[D_1]}$, then it holds that
	\begin{equation*}
		\left| \inner{\tilde{D},  \nabla^2 g(X)[\tilde{D}]} - \inner{\tilde{D},  \nabla^2 g(X)[\tilde{D}]} \right| \leq \norm{\nabla^2 g(X) - \nabla^2 g(Y)}\ff.
	\end{equation*}
	Besides, from the expression of the hessian of $\norm{X\tp X - I_p}\fs$, we achieve
	\begin{equation*}
		\begin{aligned}
			&\inner{\tilde{D}, \tilde{D}(X\tp X - I_p) + 2 X \Phi(\tilde{D}\tp X)} \leq \inner{\tilde{D}, \tilde{D}(X\tp X - I_p)} + 2 \inner{\tilde{D}, X \Phi(\tilde{D}\tp X)}\\
			\leq{}& \norm{D}\fs \norm{X\tp X - I_p}\ff + 2 \norm{\Phi( X\tp \tilde{D})}\fs = \norm{D}\fs \norm{X\tp X - I_p}\ff + 2 \norm{\Phi( (X - Y)\tp \tilde{D})}\fs \\
			\leq{}& \norm{X\tp X - I_p}\ff  + 2\norm{X-Y}\fs \norm{D}\fs \leq \norm{X\tp X - I_p}\ff  + \frac{72}{121} \norm{X\tp X - I_p}\fs\\
			\leq{}& \frac{11}{10} \norm{X\tp X - I_p}\ff. 
		\end{aligned}
	\end{equation*}	
	
	Therefore, we conclude that
	\begin{equation*}
		\begin{aligned}
			&\sigma_{\min}(\nabla h(X))\\
			={}& \min_{D \in \bb{R}^n, \norm{D}\ff = 1} \inner{D, \nabla^2 h(X)[D]}\leq \min_{D \in \ca{T}_Y, \norm{D}\ff = 1} \inner{D, \nabla^2 h(X)[D]}\\
			\leq{}& \inner{\tilde{D}, \nabla^2 h(X)\tilde{D}} = \inner{\tilde{D}, \nabla^2 g(X)\tilde{D}} + \beta \inner{\tilde{D}, \tilde{D}(X\tp X - I_p) + 2 X \Phi(\tilde{D}\tp X)}\\
			\leq{}& \sigma_{\min}(\hess f(Y)) + \norm{\nabla^2 g(X) - \nabla^2 g(Y)}\ff + \frac{11\beta}{10} \norm{X\tp X - I_p}\ff\\
			\leq{}& \sigma_{\min}(\hess f(Y)) + \norm{\nabla^2 g(X) - \nabla^2 g(Y)}\ff + \frac{9}{2}\norm{\nabla h(X)}\ff,
		\end{aligned}
	\end{equation*}
	and complete the proof.
\end{proof}

\subsection{{\L}ojasiewicz gradient inequality}
In this section, we study the relationship between the Riemannian {\L}ojasiewicz gradient inequality for $f(X)$ and the Euclidean {\L}ojasiewicz gradient inequality for $h(X)$.

\begin{prop}
	Suppose $f(X)$ satisfies the Riemannian {\L}ojasiewicz gradient inequality at $X\in \ca{S}_{n,p}$ with {\L}ojasiewicz exponent $\theta \in \left(0, \frac{1}{2}\right]$, i.e there exists a neighborhood $U \subset \ca{S}_{n,p}$ and a constant $C > 0$ such that 
	\begin{equation*}
		\norm{\grad  f(Y)}\ff  \geq C |f(Y) - f(X)|^{1-\theta},
	\end{equation*} 
	holds for any $Y \in U$ and the penalty parameter of \ref{Prob_Pen} satisfies $\beta > \max\{8CM_1, 1, \bar{\beta} \}$. Then $h(X)$ satisfies the {\L}ojasiewicz gradient inequality at  $X \in \ca{S}_{n,p}$ with {\L}ojasiewicz exponent $\theta \in (0, \frac{1}{2}]$.
\end{prop}
\begin{proof}
	For any $Y \in \Omega$, we denote $Z := Y(Y\tp Y)^{-\frac{1}{2}}$. It is clear that $Z \in \ca{S}_{n,p}$. By Lemma \ref{Le_gradient_h} and the Riemannian {\L}ojasiewicz gradient inequality of $f$, we have
	\begin{equation*}
		\norm{	\nabla g(Z)}\ff = \norm{\grad f(Z)}\ff \geq C|f(Z) - f(X)|^{1-\theta} =  C|g(Z) - g(X)|^{1-\theta}. 
	\end{equation*}
	Besides, since
	\begin{equation*}
		\begin{aligned}
			{}&\norm{\Apen{X}^2X\tp X  - I_p}\fs = \norm{X\tp X - I_p + (I_p - X\tp X) X\tp X + \frac{1}{4}X\tp X(X\tp X - I_p)^2 }\ff\\
			={}& \norm{ \left(\frac{1}{4}X\tp X - I_p \right)(X\tp X - I_p)^2 }\ff,
		\end{aligned}
	\end{equation*}	
	we obtain
	\begin{equation}
		\label{Eq_xnc1}
		| g(Y) -  g(Z)| \leq M_1\norm{Y\tp Y - I_p}\fs.
	\end{equation}
	Together with Lemma \ref{Le_grad_f_upper_bound}, we can conclude that
	\begin{equation*}
		\norm{\nabla h(Y)}\ff \geq \frac{1}{2}\norm{\nabla g(Z)}\ff + \frac{\beta}{4} \norm{Y\tp Y - I_p}\ff. 
	\end{equation*}
	
	In addition, since $\theta \in (0, \frac{1}{2}]$, and $Y \in \Omega$, we have
	\begin{equation*}
		\begin{aligned}
			& |h(Y) - h(X)|^{1-\theta} =  \left|g(Y) - g(X) + \frac{\beta}{4}\norm{Y\tp Y - I_p}\fs \right|^{1-\theta}\\
			\leq{}& \left|g(Y) - g(X)  \right|^{1-\theta} + \left(\frac{\beta}{4}\norm{Y\tp Y - I_p}\fs\right)^{1-\theta}
			\leq{} \left|g(Y) - g(X)  \right|^{1-\theta} + \frac{\beta}{4}\norm{Y\tp Y - I_p}\ff. 
		\end{aligned}
	\end{equation*}
	
	Therefore, we have
	\begin{equation*}
		\begin{aligned}
			&\norm{	\nabla h(Y)}\ff
			\geq {} \norm{\nabla g(Z)}\ff  + \frac{\beta}{4}  \norm{Y\tp Y - I_p}\ff 
			\geq {} C| g(Z) -  g(X)|^{1-\theta} + \frac{\beta}{4} \norm{Y\tp Y - I_p}\ff \\
			\overset{(i)}{\geq} {}& C| g(Y) -  g(X)|^{1-\theta} - C|g(Z) - g(Y)|^{1-\theta} +\frac{\beta}{4} \norm{Y\tp Y - I_p}\ff\\
			\overset{(ii)}{\geq} {}& C| g(Y) -  g(X)|^{1-\theta} - C M_1^{1-\theta}\norm{Y\tp Y - I_p}\ff^{2-2\theta}   + \frac{\beta}{4} \norm{Y\tp Y - I_p}\ff \\
			\geq {}& C| g(Y) -  g(X)|^{1-\theta} + \frac{\beta}{8}\norm{Y\tp Y - I_p}\ff
			\geq{}  \min\left\{C, \frac{1}{2}\right\} |h(Y) - h(X)|^{1-\theta}.
		\end{aligned}
	\end{equation*}
	Here inequality $(i)$ uses the fact that $  |a|^{1-\theta} + |b|^{1-\theta} \geq (|a| + |b|)^{1-\theta} \geq |a+ b|^{1-\theta}$ for any $a, b \in \bb{R}$, $\theta \in (0, \frac{1}{2}]$. 
	Besides, inequality $(ii)$ directly follows from \eqref{Eq_xnc1}.  
	As a result, we obtain
	\begin{equation*}
		\norm{	\nabla h(Y)}\ff \geq \min\left\{C, \frac{1}{2}\right\}  |h(Y) - h(X)|^{1-\theta},
	\end{equation*}
	which concludes the proof. 
\end{proof}

Besides, we can even show that \ref{Prob_Ori} and \ref{Prob_Pen} 
share the same local minimizers.

\begin{theo}
	\label{The_local_minimizer}
	Suppose Assumptions \ref{Assumption_global_f} and \ref{Assumption_second_order} hold, and $\beta \geq \hat{\beta}$, then \ref{Prob_Pen} and \ref{Prob_Ori} share the same local minimizers. 
\end{theo}
\begin{proof}
	By Corollary \ref{Coro_Equivalence} and Corollary \ref{The_Equivalence_global_second_order}, we can conclude that any local minimizers of \ref{Prob_Pen} are on Stiefel manifold. 
	Since  $h(X) = f(X)$ holds for any $X \in \ca{S}_{n,p}$, then any local minimizers of \ref{Prob_Pen} are local minimizers of \ref{Prob_Ori}.  
	
	On the other hand, let $X^* \in \ca{S}_{n,p}$ be a local minimizer of \ref{Prob_Ori}, then there exists $\gamma \in \left(0, \frac{1}{12}\right)$ such that $f(Z) \geq f(X)$ holds for any $Z \in \ca{S}_{n,p}$, $\norm{Z - X^*}\ff \leq \gamma$. Then for any $Y \in \bb{R}^{n\times p}$, $\norm{Y - X^*}\ff \leq \frac{\gamma}{2}\gamma$, $ Y\in\BOM_{11\gamma/12} $, we can obtain that 
	\begin{equation*}
		\norm{\ProjS(Y) - X^*}\ff \leq \norm{\ProjS(Y) - Y}\ff + \norm{Y-X^*}\ff \leq \frac{\gamma}{2} + \frac{\gamma}{2} \leq  \gamma. 
	\end{equation*}
	Here the second inequality recalls the relationship \eqref{Eq_Le_grad_f_upper_bound}.
	Then it follows from Proposition \ref{Le_ortho_fval} that 
	\begin{equation*}
		\begin{aligned}
			h(Y) - h(X^*) ={}& h(Y) - h(\ProjS(Y)) + h(\ProjS(Y)) - h(X^*) 
			\geq{} \left( \frac{\beta}{4} -  \frac{M_1}{2} \right) \norm{Y\tp Y - I_p}\fs \geq 0,
		\end{aligned}
	\end{equation*}
	which concludes the proof. 
\end{proof}

\section{Application}
\subsection{Theoretical analysis for nonlinear conjugate gradient method}
Nonlinear conjugate gradient (CG) methods are a class of important methods for solving unconstrained optimization problems. The first CG method, called Fletcher-Reeves CG (FR-CG), was proposed in \cite{fletcher1964function}. Then various nonlinear CG methods were developed \cite{hestenes1952methods,fletcher1964function,daniel1967conjugate,polyak1969conjugate,fletcher2013practical,liu1991efficient,dai1999nonlinear,hager2005new}. Interested readers can refer to the survey \cite{hager2006survey} for details.
Recently, a number of works extend CG methods to optimization problems over the Stiefel manifold \cite{abrudan2009conjugate,sato2015new,sato2016dai,zhu2017riemannian}. These works are developed within the frameworks provided by \cite{Absil2009optimization}, and thus extensively involve retractions and parallel or vector transports. Hence, as discussed before, we are forced to comply with the low efficiency if using the parallel transports or the lack of convergence if choosing the vector transport instead.

Contrarily, applying nonlinear CG methods to minimizing \Expens over $\bb{R}^{n\times p}$  
can inherit both of the efficiency and convergence properties directly.
%
The exact penalty model \ref{Prob_Pen} provides a bridge between the unconstrained optimization approaches and the original model \ref{Prob_Ori}.
In this section, we demonstrate the power of this ``bridge" through directly applying
the FR-CG method to solve \ref{Prob_Ori} through \ref{Prob_Pen}.
First of all, we present an \ref{Prob_Pen} version of the FR-CG in Algorithm \ref{Alg:CG}.
\begin{algorithm}[htbp]
	\begin{algorithmic}[1]   
		\Require Input data:  functions $f$.
		\State Choose initial guess $X_0$ and parameters $0< \delta \leq \sigma \leq \frac{1}{2}$, set  $k:=0$, $D_0 = -\nabla h(X_0)$.
		\While{not terminate}
		\State Compute the stepsize $\eta_k$ by strong Wolfe line search \cite{hager2006survey}:
		\begin{align*}
			h(\Xk + \eta_k D_k) - h(\Xk) &\leq \delta \eta_k \inner{\nabla h(\Xk), D_k},\\
			|\inner{\nabla h(\Xk + \eta_k D_k), D_k}| &\leq -\sigma \inner{\nabla h(\Xk), D_k}.
		\end{align*}
		\State $\Xkp = \Xk + \eta_k D_k$.
		\State Compute the CG update parameter $\tau_k = \frac{\norm{\nabla h (\Xkp)}\fs }{\norm{\nabla h(\Xk)}\fs}$ .
		\State Compute the search direction $D_{k+1} = -\nabla h(\Xkp) + \tau_k D_k$.
		\State Set $k:=k+1$.
		\EndWhile
		\State Return $X_k$.
	\end{algorithmic}  
	\caption{Nonlinear FR-CG method for solving \ref{Prob_Pen}.}  
	\label{Alg:CG}
\end{algorithm}
To prove the convergence of Algorithm \ref{Alg:CG}, we first illustrate a nice property of \ref{Prob_Pen} through the following lemma.

\begin{lem}
	\label{Le_level_bounded}
	Suppose $\beta \geq 384M_0$. For any sequence $\{\Xk\}$ that satisfies $X_0 \in \BOMtiny$, $h(\Xk) \leq h(X_0)$ and $\norm{\Xkp - \Xk}\ff \leq \frac{1}{24}$ for any $k\geq 0$. Then it holds that $\{\Xk\} \subset \BOMt$.
\end{lem}
\begin{proof}
	Firstly, for any $Y\in\BOMtiny$ and 
	$Z\in\Omega\setminus\BOMt$, we have
	\begin{equation}
		\label{Eq_Le_level_bounded_0}
		h(Y) - h(Z) < \sup_{W \in \Omega } h(W) - \inf_{W \in \Omega } h(W) + \frac{\beta}{1152}- \frac{\beta}{288} \leq M_0 -\frac{\beta}{384} \leq 0. 
	\end{equation}
	Then we prove the lemma by induction. Suppose $\{X_0, ...., X_k\} \subset \BOMt$. Then notice that $\norm{\Xkp - \Xk}\ff \leq \frac{1}{24}$, it holds that $\Xkp \in \BOMs$. Together with the fact that $h(\Xkp) \leq h(X_0)$, it directly follows from \eqref{Eq_Le_level_bounded_0} that $\Xkp \in \BOMt$. Therefore, the induction illustrates that $\{\Xk\} \subset \BOMt$, thus we complete the proof. 
\end{proof}

Lemma \ref{Le_level_bounded} guarantees that the iterates generated by any monotonic algorithm
starting from an initial point $X_0\in\BOMt$ are restricted in the region $\BOMs$ under mild conditions. Then the Step 3 in Algorithm \ref{Alg:CG} indicates that $\{h(X_k)\}$ is monotone decreasing. Then combining Lemma \ref{Le_level_bounded} and Assumption \ref{Assumption_local_f}, we conclude that the objective function $f$ satisfies the Lipschitz conditions and boundness conditions in \cite{hager2006survey}. Therefore, based on the Zoutendijk condition \cite{zoutendijk1970nonlinear}  and \cite[Theorem 4.2]{hager2006survey}, we can directly establish the follow global convergence result for Algorithm \ref{Alg:CG} and omit its proof. 
\begin{theo}
	\label{The_Convergence_CG}
	Suppose $\beta \geq \max \{384M_0, \hat{\beta}\}$. Let
	$\{(X_k,D_k)\}$ be the sequence generated by Algorithm \ref{Alg:CG} initiated from 
	$X_0\in\BOMtiny$. If 
	  the inequalities
		$$\inner{\nabla h(\Xk), D_k} <0  \quad  \text{and} \quad  \eta_k\norm{D_k}\ff \leq \frac{1}{24}$$
	hold for all $k\geq 0$ and  we further have
	\begin{eqnarray}\label{add1}
		\sum_{k = 0}^{+\infty} \frac{\inner{\nabla h(\Xk), D_k}^2}{\norm{D_k}\fs} < +\infty,
	\end{eqnarray}
	then
     
	any accumulation point of $\{\Xk\}$ is a first-order stationary point of \ref{Prob_Ori}.
\end{theo}

\begin{rmk}
	Algorithm \ref{Alg:CG} and Theorem \ref{The_Convergence_CG} 
	take FR-CG as example. In fact, if we update the parameter sequence $\{\tau_k\}$ 
	by the PRP \cite{polyak1969conjugate}, DY \cite{dai1999nonlinear}, or HS \cite{hestenes1952methods} formulas,
	we can obtain similar global convergence properties as well. 
	Interested readers are referred 
	to the survey paper \cite{hager2006survey} for details. 
\end{rmk}

\begin{rmk}
	The condition \eqref{add1} in Theorem \ref{The_Convergence_CG} refers to the Zoutendijk condition \cite{zoutendijk1970nonlinear,hager2006survey}, which is sufficient for the global convergence properties for a great number of nonlinear conjugate gradient methods.  
\end{rmk}
	
\begin{rmk}
	It is worth mentioning that a small penalty parameter $\beta$ may  lead to the failure of convergence, while a large penalty parameter may result in a large condition number, thus lead to slow convergence rate. Several existing works on developing penalty methods for
	\ref{Prob_Ori} have suggested some practically useful choice of the penalty parameter $\beta$,  interested readers could refer to \cite{gao2019parallelizable,xiao2020class,xiao2020l21} for details. 
\end{rmk}

\subsection{Numerical experiments}
In this section, we numerically demonstrate the power of the ``bridge", provided by ExPen, between the unconstrained optimization approaches and the original model OCP.
All the numerical experiments in this section are run in serial in a platform with AMD Ryzen 5800H CPU and 16GB RAM under Ubuntu 18.10 running Python 3.7.0 and Numpy 1.20.0 \cite{numpy2020array}. 

We choose nonlinear eigenvalue problem and the Brockett function minimization as the test problems. The details of
how to construct the test instances are described in the following two subsections, respectively.

In the presented experiments, we set the penalty parameter $\beta$ in \ref{Prob_Pen} as suggested in \cite{gao2019parallelizable,xiao2020class,xiao2020l21}, i.e. $\beta = \norm{\nabla f(X_0)}\ff/10$, where $X_0$ is the initial point. Besides, we choose the nonlinear conjugate gradient solver \cite{fletcher1964function,nocedal2006numerical} provided in the package SciPy 1.6.3  \cite{scipy2020SciPyNMeth} to minimize \ref{Prob_Pen} in $\bb{R}^{n\times p}$. This optimization approach is referred as ExPen-CG. We terminate ExPen-CG when $\norm{\nabla h(\Xk)}\ff \leq 10^{-3}$, or the maximum number of iterations exceeds $10000$, while keeping all the other parameters as their default values
in the package. 

For comparison,  we first select the Riemannian conjugate gradient (RCG) solver from the package PyManopt (version 0.2.5) \cite{townsend2016pymanopt}. RCG is one of the state-of-the-art Riemannian solvers in Python platform. Furthermore, we also choose some state-of-the-art infeasible optimization solvers into the comparisons. These solvers include PCAL \cite{gao2019parallelizable}, PenC \cite{xiao2020class} and SLPG \cite{xiao2021penalty} from the STOP package \cite{stop2022website}.  We terminate these solvers when the maximum number of iterations exceeds $10000$ and set the tolerance for gradient as $10^{-3}$.  Meanwhile, we set the other parameters in these solvers by default. Furthermore, for the final solution $\tilde{X}$ generated by all the compared algorithms, we project $\tilde{X}$ onto the Stiefel manifold as the post-processing step employed in  \cite{gao2019parallelizable,xiao2020class,xiao2020l21,xiao2021penalty}.  

\subsubsection{Nonlinear eigenvalue problems}
In this subsection, we test the performance of all the compared solves in solving a class of nonlinear eigenvalue problems arisen from electronic structure calculation \cite{liu2014Convergence,yang2009Convergence,cai2018eigenvector},
\begin{equation}\label{neg}
	\begin{aligned}
		\min_{X \in \bb{R}^{n\times p}} \quad &f(X) = \frac{1}{2} \tr\left( X\tp LX \right) + \frac{\alpha}{4} \rho_X\tp L^{\dagger} \rho_X\\
		\text{s. t.} \quad & X\tp X = I_p,
	\end{aligned}
\end{equation}
where $\rho_X:= \mathrm{diag}(XX\tp)$, $L$ is a tridiagonal matrix with $2$ as diagonal entries and $-1$ as subdiagonal entries. Besides, $L^\dagger$ refers to the pseudo-inverse of $L$. 
We initiate all the compared solvers at the same initial point, which is randomly generated over $\ca{S}_{n,p}$. Table \ref{Table_noeig_n} and Table \ref{Table_noeig_p} illustrate the performance of all the compared algorithms  in solving problem \ref{neg}
with different combinations of problem parameters $n$, $p$. Here, we run each instance for $10$ times and present the averaged results. We can learn from Table \ref{Table_noeig_n} and Table \ref{Table_noeig_p}  
that all the compared solvers reach similar function values while ExPen-CG is comparable with the state-of-the-art solvers. Remarkably, ExPen-CG outperforms RCG in terms of CPU time and iterations.

\begin{longtable}{@{}cccccc@{}}
	\toprule
	Solver                        & Fval    & Iteration & Stationarity & Feasibility & CPU time(s) \\* \midrule
	\endfirsthead
	\multicolumn{6}{c}%
	{{\bfseries Table \thetable\ continued from previous page}} \\
	\endhead
	\bottomrule
	\endfoot
	\endlastfoot
	\multicolumn{6}{c}{$n= 250 , p= 50$}    \\ \midrule 
	\multicolumn{1}{c|}{ExPen-CG} & 2.810709e+03 & 567.5 & 5.63e-04 & 1.79e-14 &  1.16  \\  
	\multicolumn{1}{c|}{PCAL} & 2.810709e+03 & 1528.3 & 8.44e-04 & 9.83e-15 &  1.54  \\  
	\multicolumn{1}{c|}{PenC} & 2.810709e+03 & 1211.4 & 8.00e-04 & 1.01e-14 &  1.21  \\  
	\multicolumn{1}{c|}{SLPG} & 2.810709e+03 & 1275.3 & 8.40e-04 & 9.97e-15 &  1.25  \\  
	\multicolumn{1}{c|}{RCG} & 2.810709e+03 & 1103.1 & 9.78e-04 & 5.49e-15 &  4.28  \\ \hline 
	\multicolumn{6}{c}{$n= 500, p= 50$}    \\ \midrule 
	\multicolumn{1}{c|}{ExPen-CG} & 2.810709e+03 & 632.6 & 6.83e-04 & 1.92e-14 &  3.52  \\  
	\multicolumn{1}{c|}{PCAL} & 2.810709e+03 & 1739.5 & 8.67e-04 & 1.00e-14 &  4.64  \\  
	\multicolumn{1}{c|}{PenC} & 2.810709e+03 & 1486.1 & 9.04e-04 & 9.99e-15 &  3.82  \\  
	\multicolumn{1}{c|}{SLPG} & 2.810709e+03 & 1274.7 & 8.16e-04 & 9.77e-15 &  3.67  \\  
	\multicolumn{1}{c|}{RCG} & 2.810709e+03 & 1111.7 & 9.78e-04 & 5.85e-15 &  9.96  \\ \hline 
	\multicolumn{6}{c}{$n= 1000, p= 50$}    \\ \midrule 
	\multicolumn{1}{c|}{ExPen-CG} & 2.810709e+03 & 715.6 & 9.35e-04 & 2.06e-14 &  5.14  \\  
	\multicolumn{1}{c|}{PCAL} & 2.810709e+03 & 1849.6 & 9.72e-04 & 1.07e-14 &  7.61  \\  
	\multicolumn{1}{c|}{PenC} & 2.810709e+03 & 1580.9 & 9.05e-04 & 1.09e-14 &  6.18  \\  
	\multicolumn{1}{c|}{SLPG} & 2.810709e+03 & 1303.6 & 7.69e-04 & 1.11e-14 &  6.12  \\  
	\multicolumn{1}{c|}{RCG} & 2.810709e+03 & 1492.2 & 9.82e-04 & 7.80e-15 &  18.53  \\ \hline 
	\multicolumn{6}{c}{$n= 1500, p= 50$}    \\ \midrule 
	\multicolumn{1}{c|}{ExPen-CG} & 2.810709e+03 & 787.0 & 7.88e-04 & 2.22e-14 &  7.29  \\  
	\multicolumn{1}{c|}{PCAL} & 2.810709e+03 & 1680.7 & 9.24e-04 & 1.13e-14 &  9.70  \\  
	\multicolumn{1}{c|}{PenC} & 2.810709e+03 & 1651.8 & 9.29e-04 & 1.10e-14 &  9.00  \\  
	\multicolumn{1}{c|}{SLPG} & 2.810709e+03 & 1281.3 & 8.54e-04 & 1.10e-14 &  8.51  \\  
	\multicolumn{1}{c|}{RCG} & 2.810709e+03 & 1206.9 & 9.83e-04 & 8.32e-15 &  20.17  \\ \hline 
	\multicolumn{6}{c}{$n= 2000, p= 50$}    \\ \midrule 
	\multicolumn{1}{c|}{ExPen-CG} & 2.810709e+03 & 866.4 & 7.52e-04 & 2.35e-14 &  10.14  \\  
	\multicolumn{1}{c|}{PCAL} & 2.810709e+03 & 1863.0 & 8.45e-04 & 1.16e-14 &  14.17  \\  
	\multicolumn{1}{c|}{PenC} & 2.810709e+03 & 1619.0 & 9.18e-04 & 1.15e-14 &  11.78  \\  
	\multicolumn{1}{c|}{SLPG} & 2.810709e+03 & 1583.1 & 8.96e-04 & 1.14e-14 &  10.93  \\  
	\multicolumn{1}{c|}{RCG} & 2.810709e+03 & 1140.2 & 9.85e-04 & 8.43e-15 &  24.49  \\ \hline 
	\bottomrule[.4mm] 
	
	\caption{The results of the nonlinear eigenvalue problems with varying $n$.}
	\label{Table_noeig_n}
\end{longtable}

\begin{longtable}{@{}cccccc@{}}
	\toprule
	Solver                        & Fval    & Iteration & Stationarity & Feasibility & CPU time(s) \\* \midrule
	\endfirsthead
	\multicolumn{6}{c}%
	{{\bfseries Table \thetable\ continued from previous page}} \\
	\endhead
	\bottomrule
	\endfoot
	\endlastfoot
	\multicolumn{6}{c}{$n= 1000, p= 10$}    \\ \midrule 
	\multicolumn{1}{c|}{ExPen-CG} & 3.570857e+01 & 165.1 & 6.52e-04 & 5.74e-15 &  0.23  \\  
	\multicolumn{1}{c|}{PCAL} & 3.570857e+01 & 331.9 & 8.36e-04 & 3.03e-15 &  0.22  \\  
	\multicolumn{1}{c|}{PenC} & 3.570857e+01 & 295.9 & 7.34e-04 & 3.42e-15 &  0.19  \\  
	\multicolumn{1}{c|}{SLPG} & 3.570857e+01 & 266.8 & 8.64e-04 & 3.30e-15 &  0.13  \\  
	\multicolumn{1}{c|}{RCG} & 3.570857e+01 & 179.7 & 9.68e-04 & 2.43e-15 &  0.43  \\ \hline 
	\multicolumn{6}{c}{$n= 1000, p= 30$}    \\ \midrule 
	\multicolumn{1}{c|}{ExPen-CG} & 6.482086e+02 & 434.7 & 7.39e-04 & 1.07e-14 &  2.11  \\  
	\multicolumn{1}{c|}{PCAL} & 6.482086e+02 & 860.0 & 7.21e-04 & 7.09e-15 &  1.74  \\  
	\multicolumn{1}{c|}{PenC} & 6.482086e+02 & 846.7 & 8.77e-04 & 7.19e-15 &  1.61  \\  
	\multicolumn{1}{c|}{SLPG} & 6.482086e+02 & 773.3 & 7.44e-04 & 7.11e-15 &  1.30  \\  
	\multicolumn{1}{c|}{RCG} & 6.482086e+02 & 596.3 & 9.90e-04 & 5.31e-15 &  5.26  \\ \hline 
	\multicolumn{6}{c}{$n= 1000, p= 50$}    \\ \midrule 
	\multicolumn{1}{c|}{ExPen-CG} & 2.810709e+03 & 752.3 & 6.95e-04 & 2.06e-14 &  5.51  \\  
	\multicolumn{1}{c|}{PCAL} & 2.810709e+03 & 1845.1 & 9.58e-04 & 1.09e-14 &  7.56  \\  
	\multicolumn{1}{c|}{PenC} & 2.810709e+03 & 1481.5 & 8.77e-04 & 1.08e-14 &  5.78  \\  
	\multicolumn{1}{c|}{SLPG} & 2.810709e+03 & 1394.7 & 7.66e-04 & 1.10e-14 &  5.07  \\  
	\multicolumn{1}{c|}{RCG} & 2.810709e+03 & 1019.1 & 9.77e-04 & 7.85e-15 &  17.62  \\ \hline 
	\multicolumn{6}{c}{$n= 1000, p= 70$}    \\ \midrule 
	\multicolumn{1}{c|}{ExPen-CG} & 7.523209e+03 & 1111.8
	
	 & 7.15e-04 & 2.34e-14 &  14.15  \\  
	\multicolumn{1}{c|}{PCAL} & 7.523209e+03 & 3612.8 & 1.61e-03 & 1.35e-14 &  19.82  \\  
	\multicolumn{1}{c|}{PenC} & 7.523209e+03 & 2928.3 & 8.91e-04 & 1.36e-14 &  15.10  \\  
	\multicolumn{1}{c|}{SLPG} & 7.523209e+03 & 2439.5 & 9.46e-04 & 1.37e-14 &  15.11  \\  
	\multicolumn{1}{c|}{RCG} & 7.523209e+03 & 2248.9 & 9.84e-04 & 1.06e-14 &  52.19  \\ \hline 
	\multicolumn{6}{c}{$n= 1000, p= 100$}    \\ \midrule 
	\multicolumn{1}{c|}{ExPen-CG} & 2.156071e+04 & 1639.7 & 7.49e-04 & 3.50e-14 &  29.80  \\  
	\multicolumn{1}{c|}{PCAL} & 2.156071e+04 & 4775.8 & 2.77e-02 & 1.74e-14 &  39.30  \\  
	\multicolumn{1}{c|}{PenC} & 2.156071e+04 & 4499.7 & 5.82e-03 & 1.76e-14 &  35.53  \\  
	\multicolumn{1}{c|}{SLPG} & 2.156071e+04 & 4642.2 & 8.88e-04 & 1.72e-14 &  34.73  \\  
	\multicolumn{1}{c|}{RCG} & 2.156071e+04 & 2871.3 & 1.20e-03 & 1.50e-14 &  102.80  \\ \hline 
	\bottomrule[.4mm] 
	
	\caption{The results of the nonlinear eigenvalue problems with varying $p$.}
	\label{Table_noeig_p}
\end{longtable}

Furthermore, we exhibit the convergence curves of ExPen-CG in the aspect of function value gap evaluated by $f(\Xk) - f(X^*)$, the stationarity $\norm{\nabla f(\Xk)}$ and the feasibility $\norm{\Xk\tp \Xk - I_p}\ff$. Here $f(X^*)$ is computed by RCG solver from PyManopt package that satisfies  $\norm{\grad f(X^*)}_F \leq 10^{12}$.  In Figure \ref{Fig_cc_noeig}, we present the curves of ExPen-CG under different combination of the parameters.  From Figure \ref{Fig:Noeig_fval_0}, \ref{Fig:Noeig_fval_1} and \ref{Fig:Noeig_fval_2}, we can observe that the sequence $\{\Xk\}$ generated by ExPen-CG achieves almost the same function values as RCG. Moreover, Figure \ref{Fig:Noeig_feas_0}, \ref{Fig:Noeig_feas_1} and \ref{Fig:Noeig_feas_2} illustrate that the sequence generated by ExPen-CG converges towards $\ca{S}_{n,p}$. 
\begin{figure}[!htbp]
	\centering
	\subfigure[$(n, p) = (250, 50)$]{
		\begin{minipage}[t]{0.33\linewidth}
			\centering
			\includegraphics[width=\linewidth]{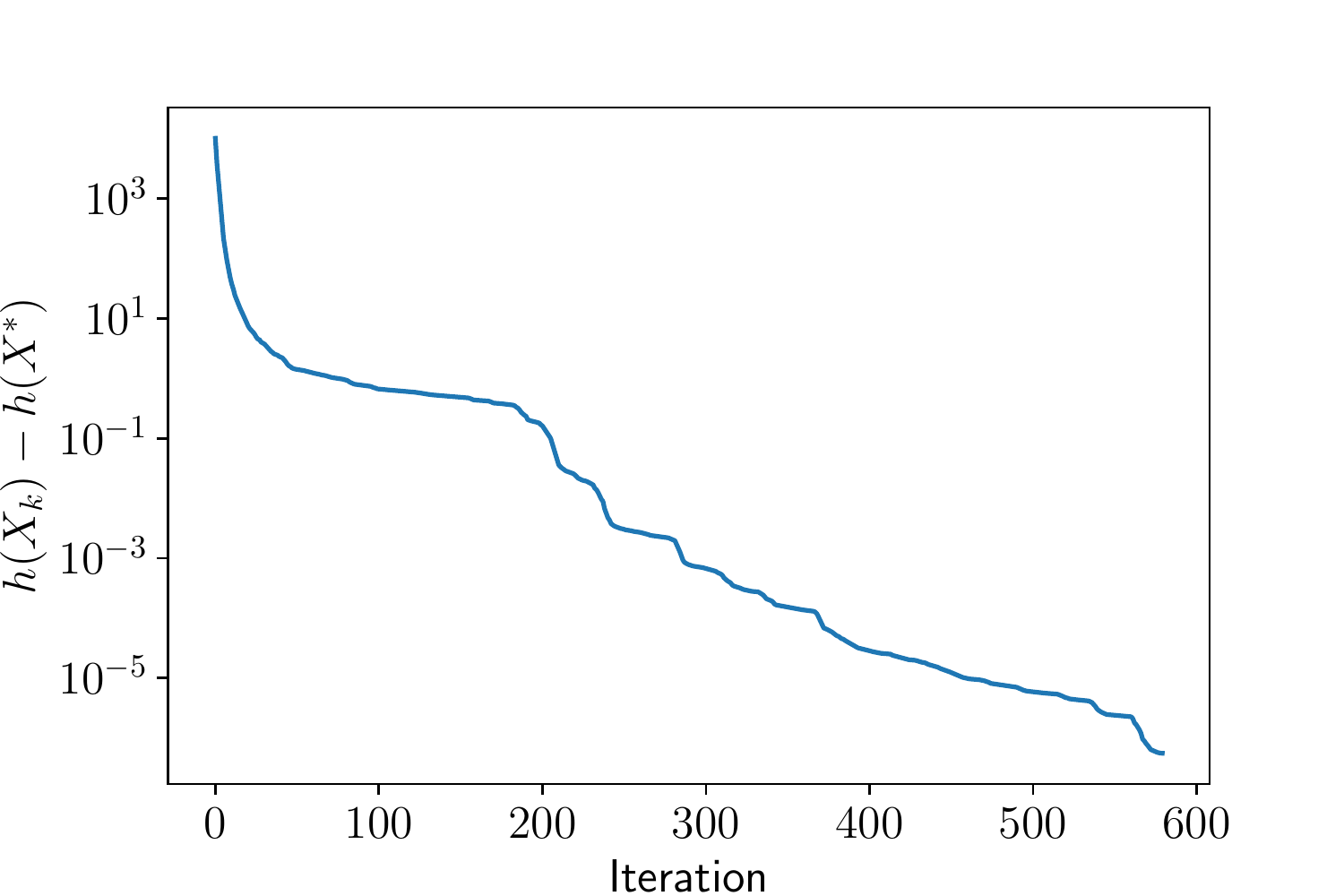}
			\label{Fig:Noeig_fval_0}
		\end{minipage}%
	}%
	\subfigure[$(n, p) = (250, 50)$]{
		\begin{minipage}[t]{0.33\linewidth}
			\centering
			\includegraphics[width=\linewidth]{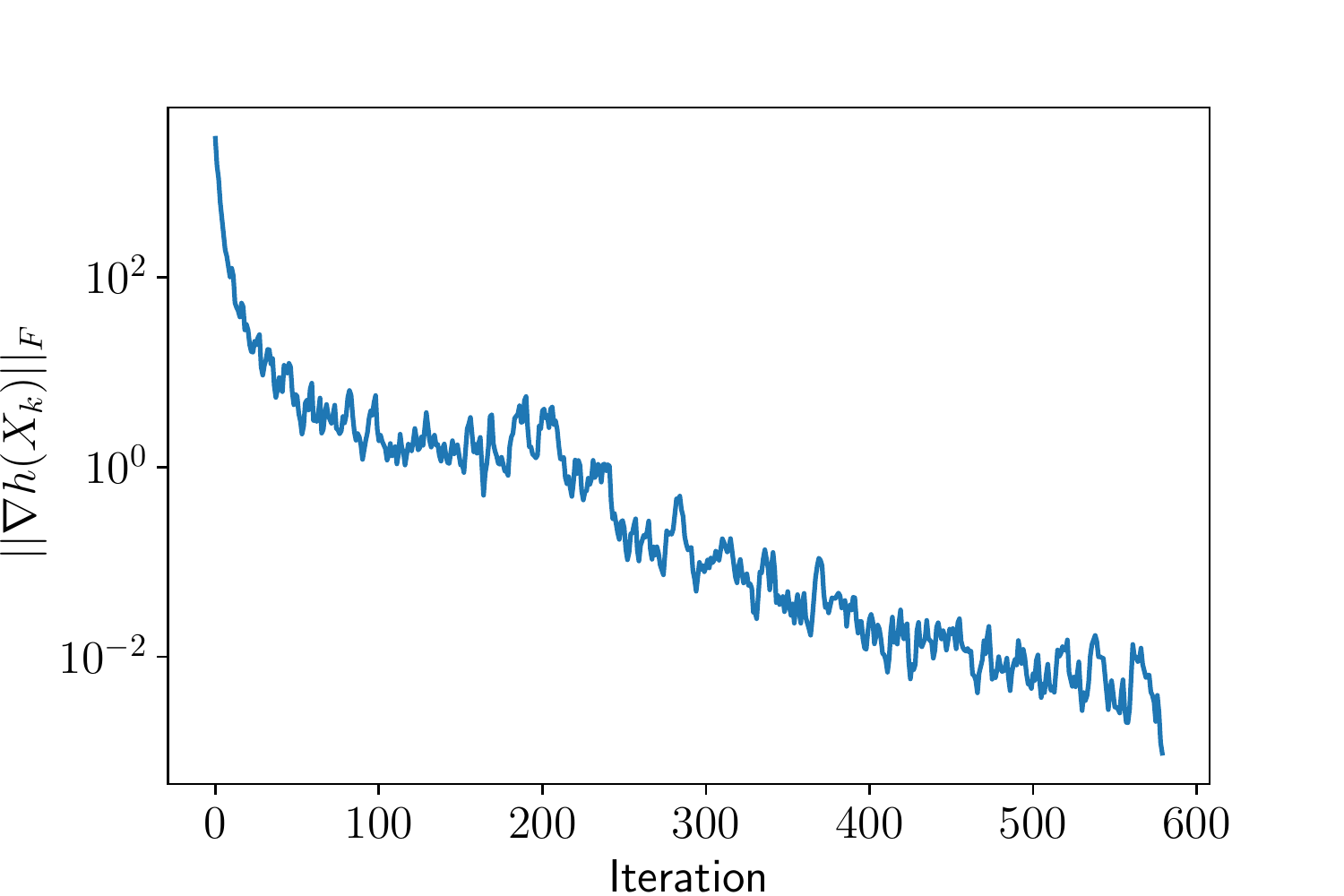}
			\label{Fig:Noeig_grad_0}
		\end{minipage}%
	}%
	\subfigure[$(n, p) = (250, 50)$]{
		\begin{minipage}[t]{0.33\linewidth}
			\centering
			\includegraphics[width=\linewidth]{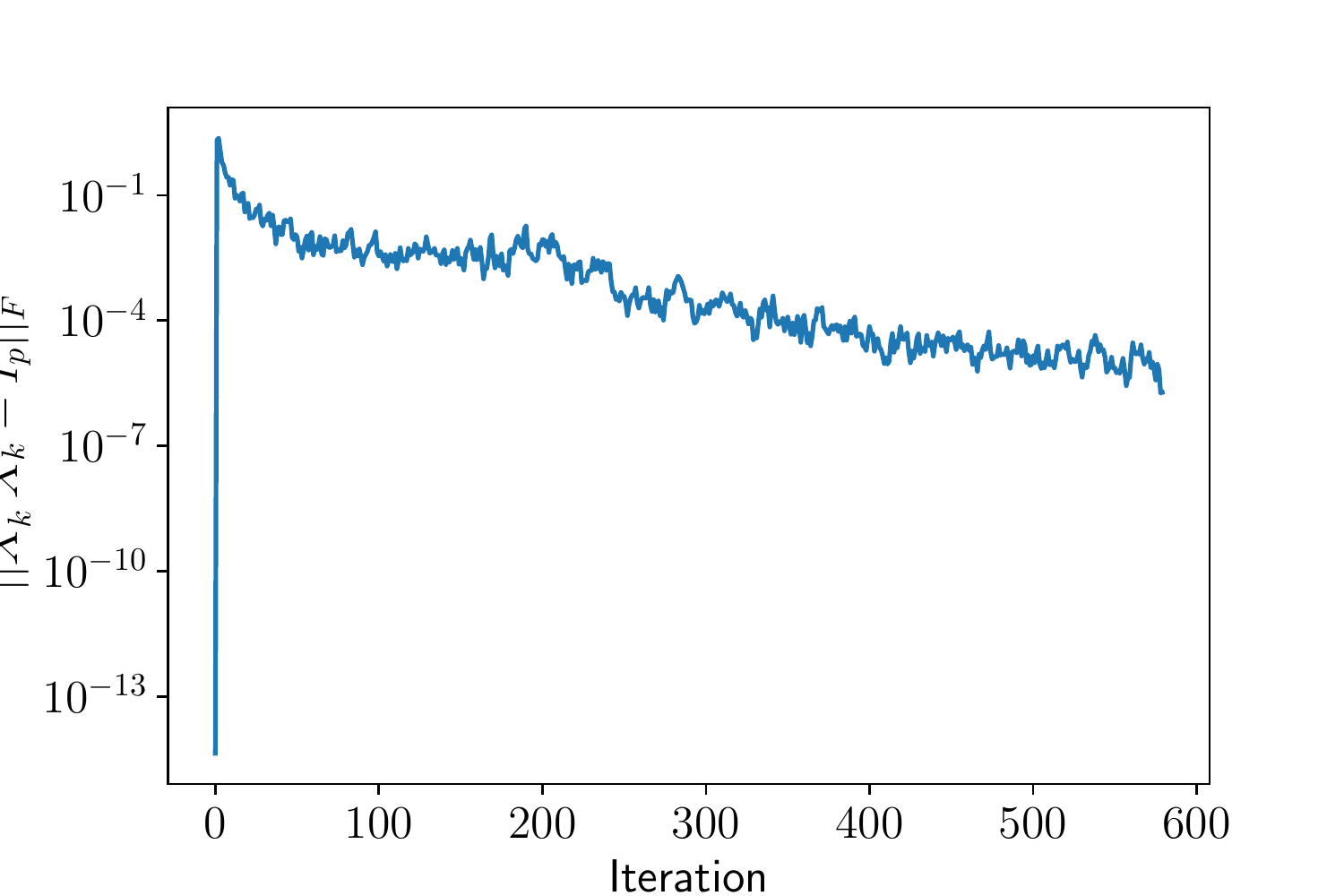}
			\label{Fig:Noeig_feas_0}
		\end{minipage}%
	}%
	
	\subfigure[$(n, p) = (2000, 50)$]{
		\begin{minipage}[t]{0.33\linewidth}
			\centering
			\includegraphics[width=\linewidth]{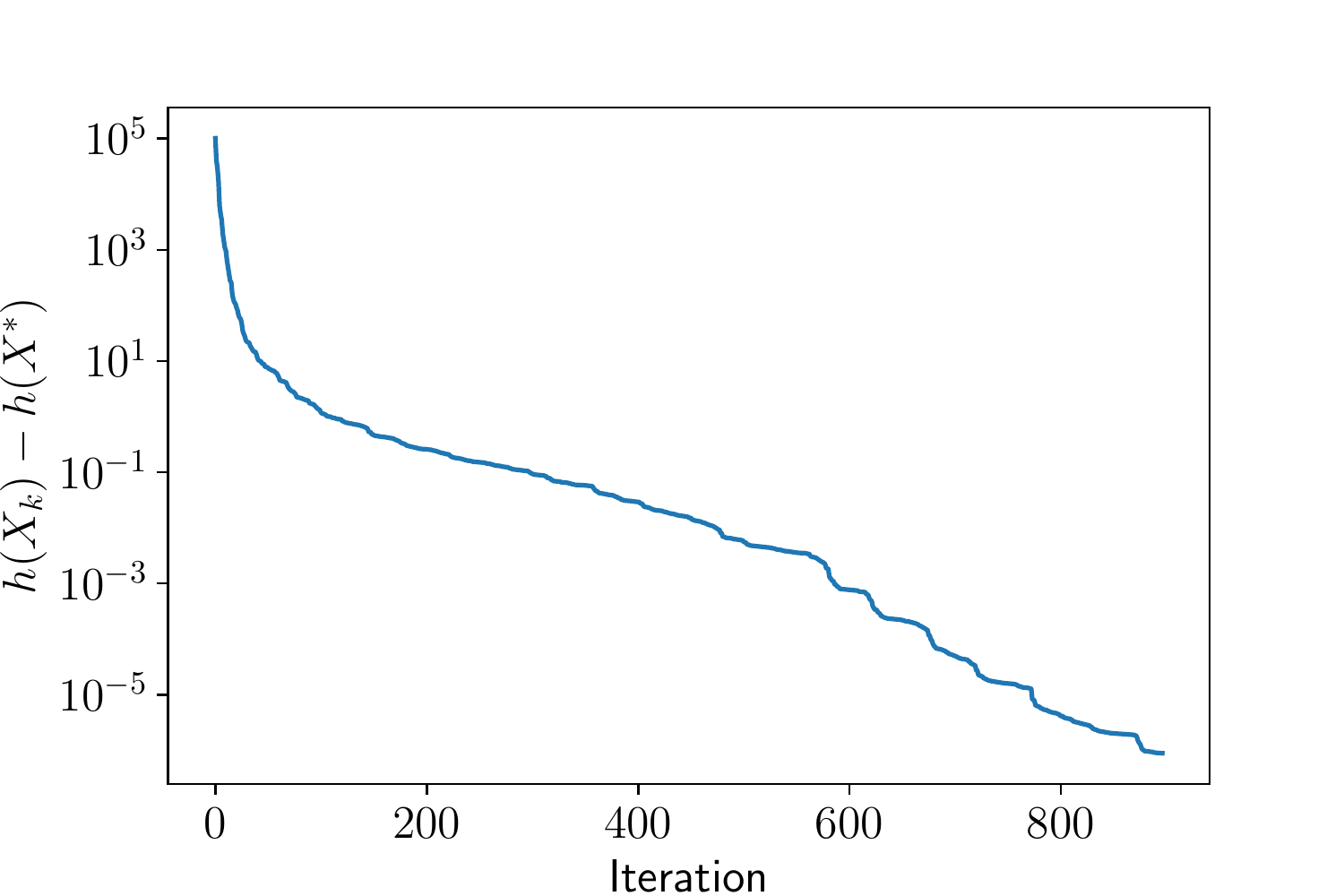}
			\label{Fig:Noeig_fval_1}
		\end{minipage}%
	}%
	\subfigure[$(n, p) = (2000, 50)$]{
		\begin{minipage}[t]{0.33\linewidth}
			\centering
			\includegraphics[width=\linewidth]{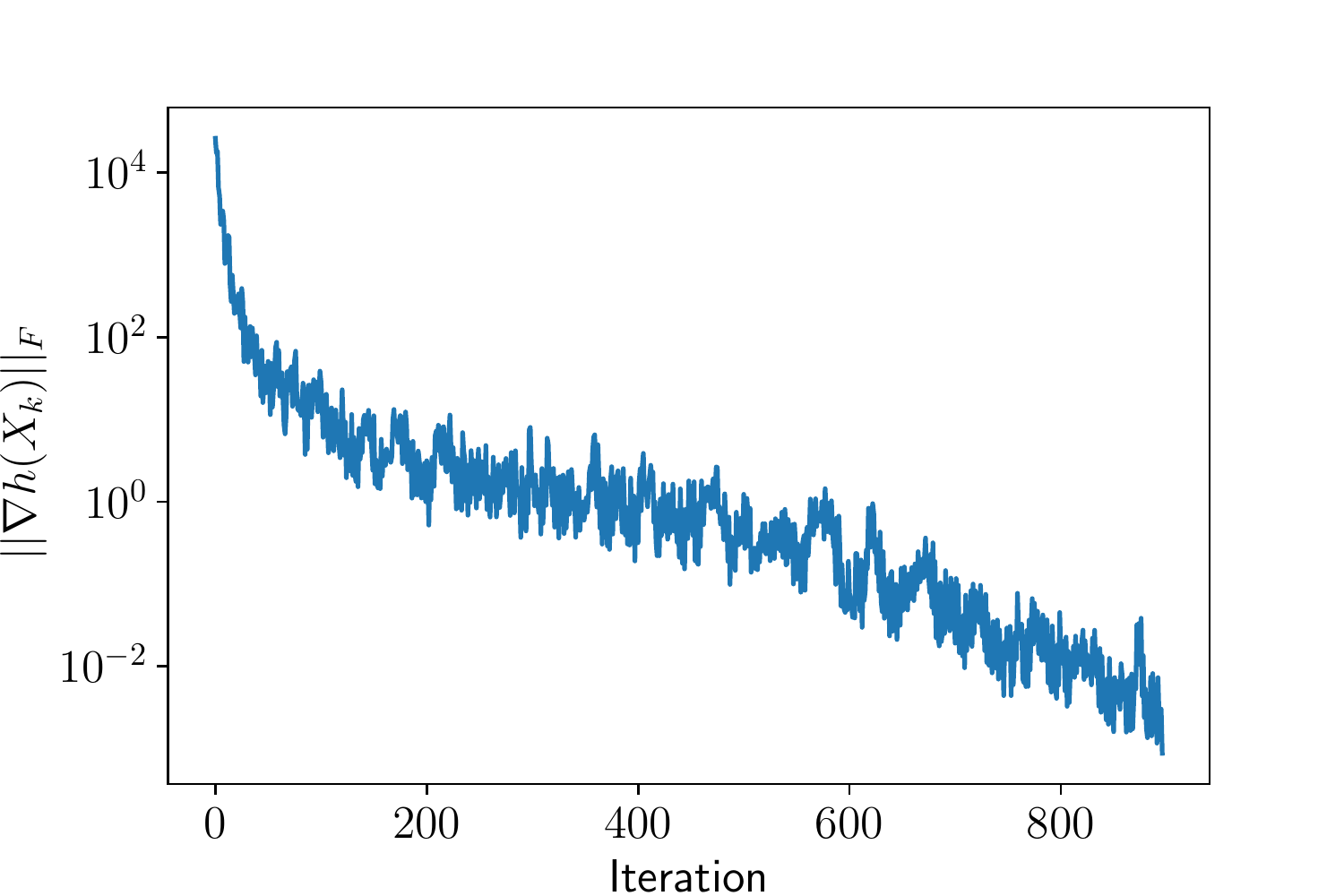}
			\label{Fig:Noeig_grad_1}
		\end{minipage}%
	}%
	\subfigure[$(n, p) = (2000, 50)$]{
		\begin{minipage}[t]{0.33\linewidth}
			\centering
			\includegraphics[width=\linewidth]{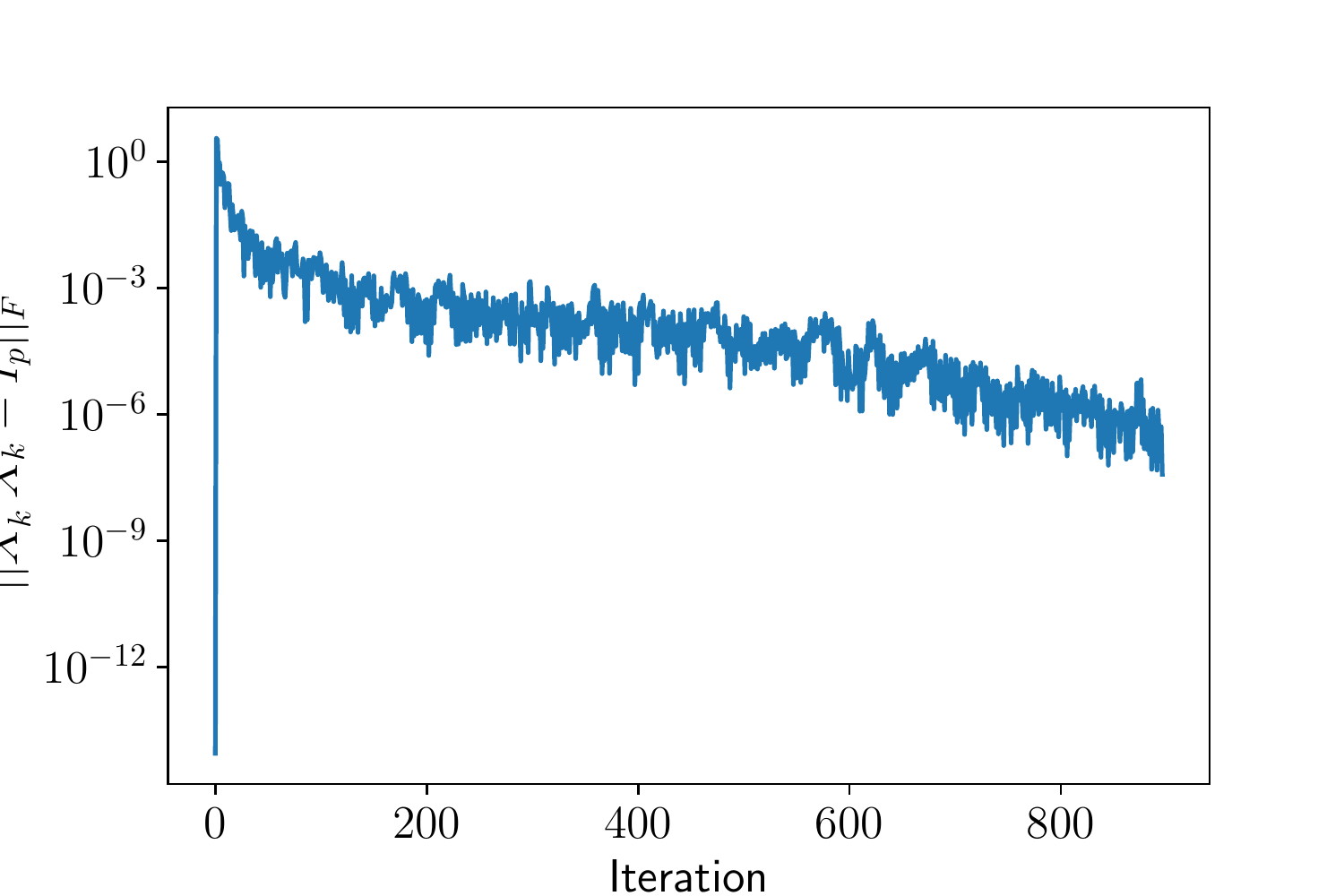}
			\label{Fig:Noeig_feas_1}
		\end{minipage}%
	}%

	\subfigure[$(n, p) = (1000, 100)$]{
		\begin{minipage}[t]{0.33\linewidth}
			\centering
			\includegraphics[width=\linewidth]{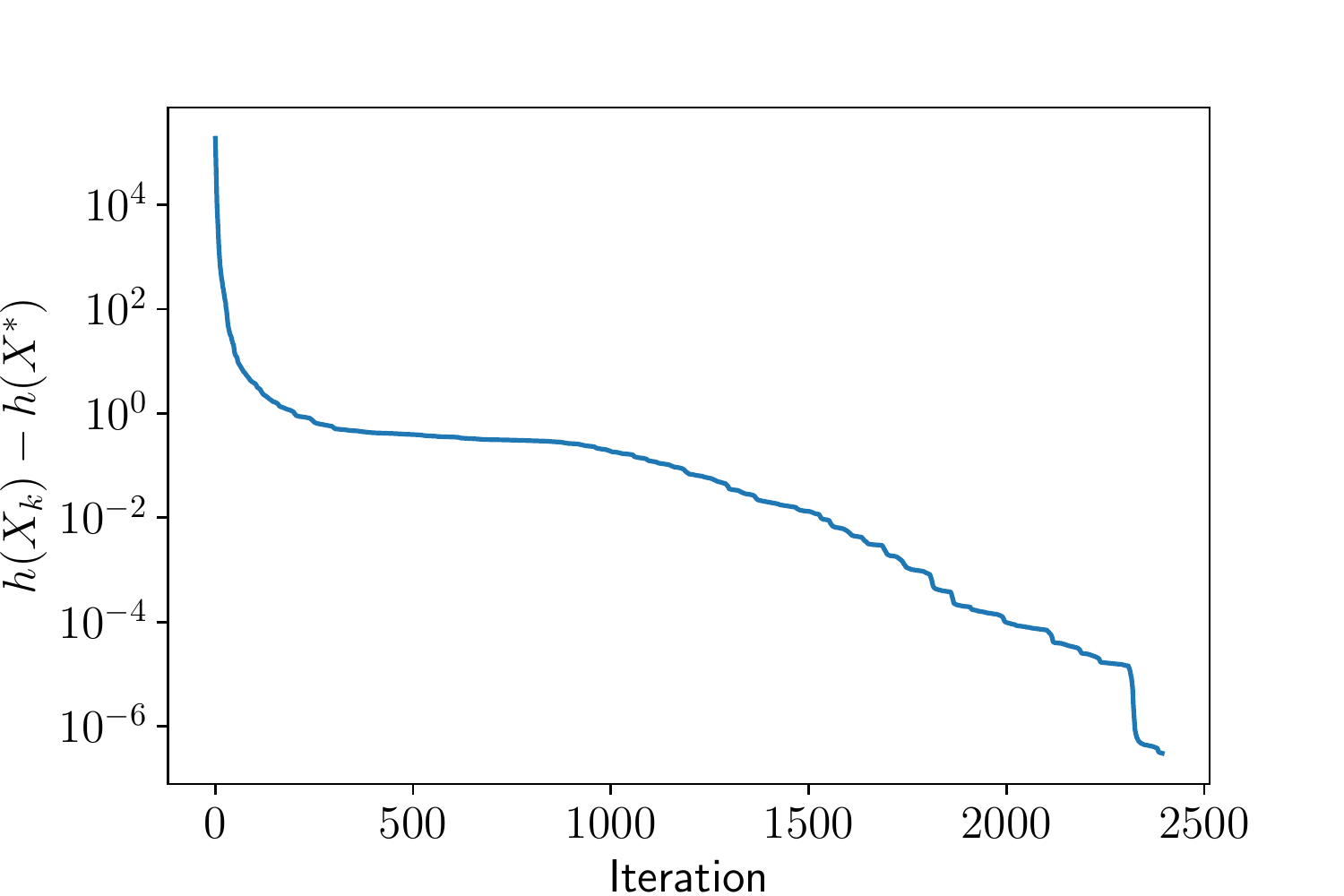}
			\label{Fig:Noeig_fval_2}
		\end{minipage}%
	}%
	\subfigure[$(n, p) = (1000, 100)$]{
		\begin{minipage}[t]{0.33\linewidth}
			\centering
			\includegraphics[width=\linewidth]{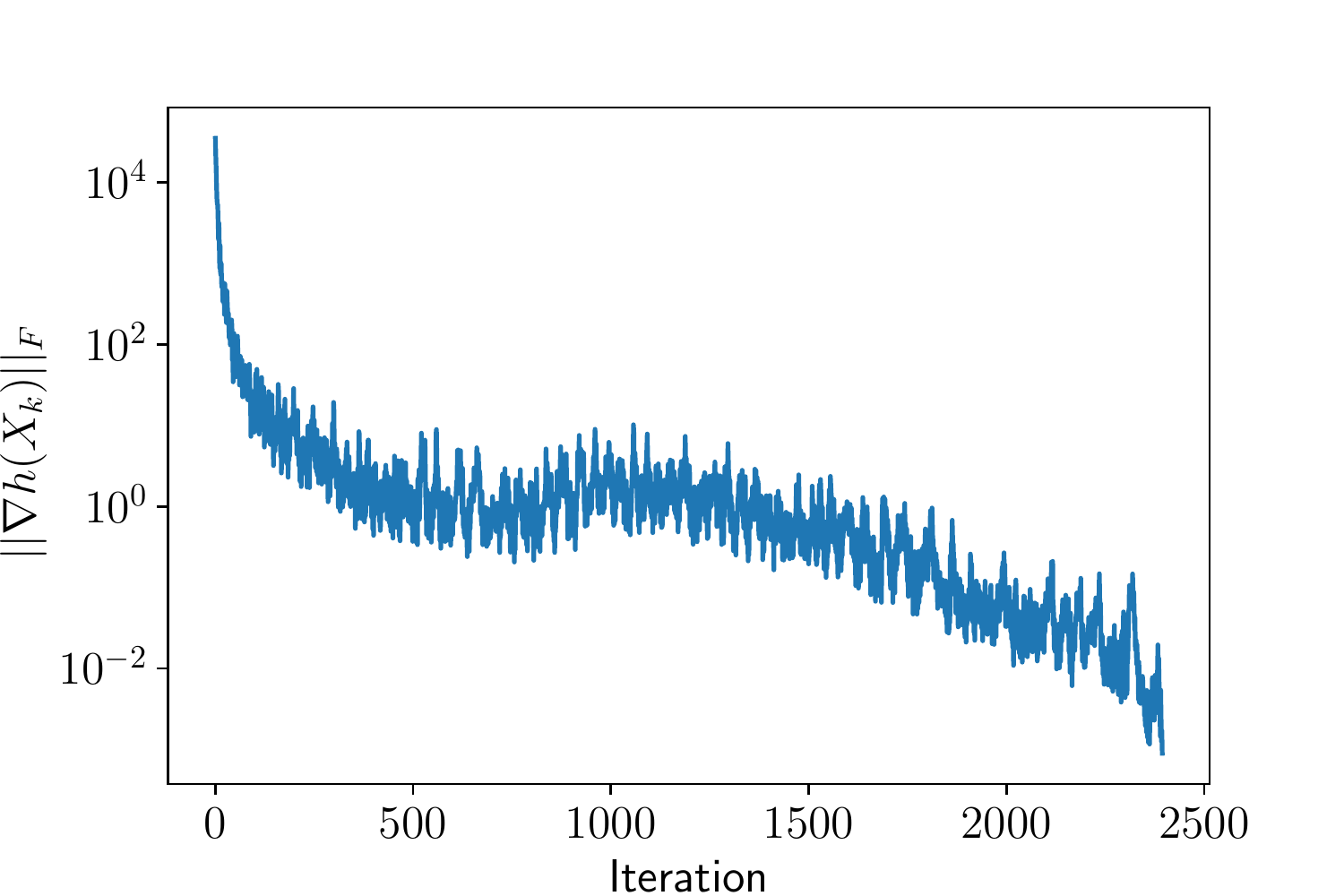}
			\label{Fig:Noeig_grad_2}
		\end{minipage}%
	}%
	\subfigure[$(n, p) = (1000, 100)$]{
		\begin{minipage}[t]{0.33\linewidth}
			\centering
			\includegraphics[width=\linewidth]{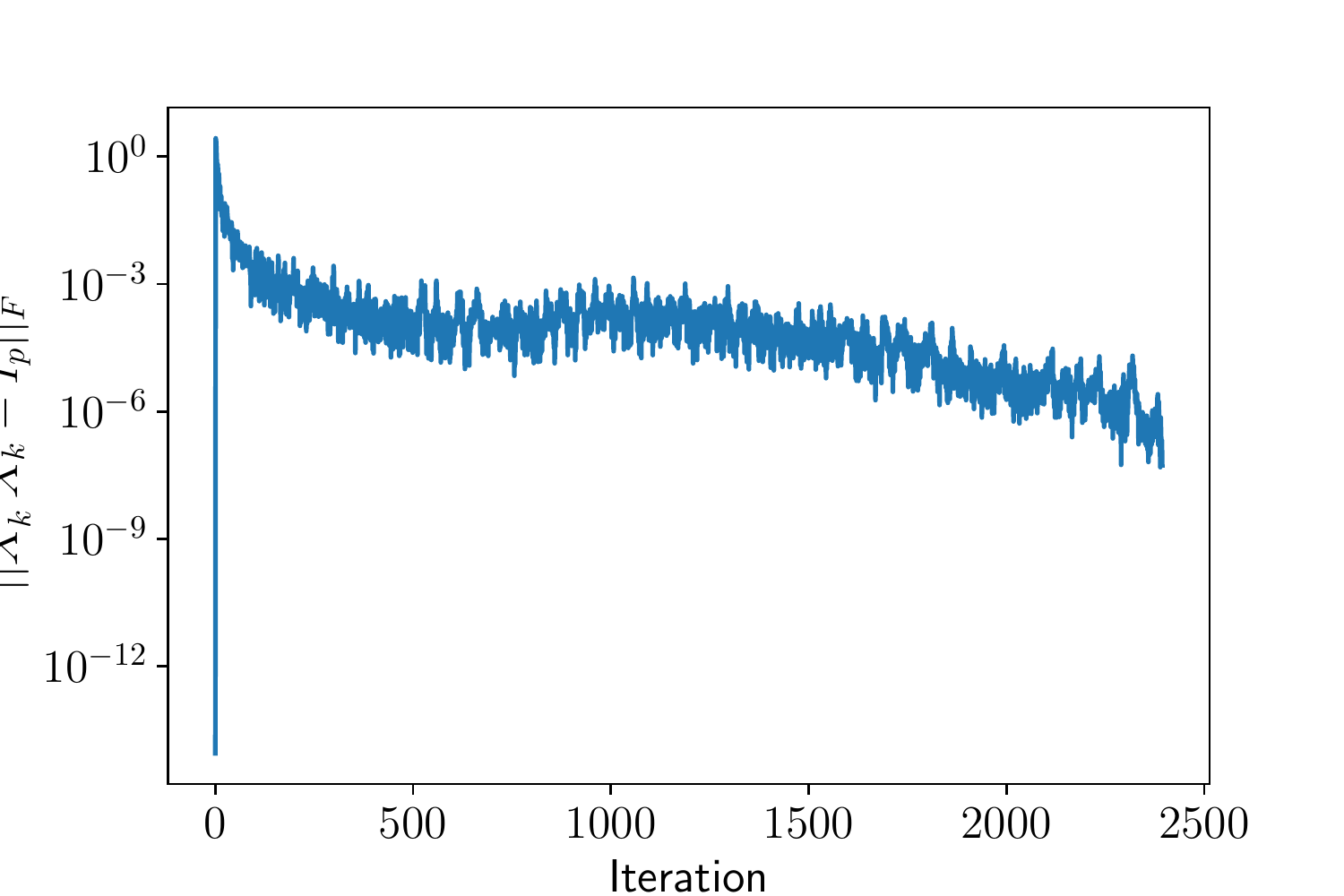}
			\label{Fig:Noeig_feas_2}
		\end{minipage}%
	}%
	\caption{The convergence curves of ExPen-CG on nonlinear eigenvalue problems.}
	\label{Fig_cc_noeig}
\end{figure}

\subsubsection{Brockett function minimization problems}
In this subsection, we test the numerical performance of all the tested algorithms on minimizing the Brockett function over the Stiefel manifold \cite{wang2020multipliers},
\begin{equation}
	\begin{aligned}
		\min_{X \in \bb{R}^{n\times p}} \quad &f(X) = \frac{1}{2} \tr\left( X\tp BXC \right)\\
		\text{s. t.} \quad & X\tp X = I_p,
	\end{aligned}
\end{equation}
where $B \in \bb{R}^{n\times n}$ and $C \in \bb{R}^{p \times p}$ are two randomly generated symmetric matrices. 
We initiate ExPen-CG and RCG at the same point, which is randomly generated over $\ca{S}_{n,p}$. Table \ref{Table_broc_n} and Table \ref{Table_broc_p} illustrate the performance of these compared algorithms with different combinations of problem parameters $n$ and $p$. Here, we run each instance for $10$ times and present the averaged results. We observe that all the compared algorithms achieve almost the same function values, and ExPen-CG achieves comparable performance as all the compared algorithms. In particular, ExPen-CG outperforms RCG in all the test instances.  

\begin{longtable}{@{}cccccc@{}}
	\toprule
	Solver                        & Fval    & Iteration & Stationarity & Feasibility & CPU time(s) \\* \midrule
	\endfirsthead
	\multicolumn{6}{c}%
	{{\bfseries Table \thetable\ continued from previous page}} \\
	\endhead
	\bottomrule
	\endfoot
	\endlastfoot
	\multicolumn{6}{c}{$n= 250, p= 50$}    \\ \midrule 
	\multicolumn{1}{c|}{ExPen-CG} & -9.611694e+00 & 962.7 & 7.97e-04 & 1.64e-14 &  2.99  \\  
	\multicolumn{1}{c|}{PCAL} & -9.610215e+00 & 2209.5 & 9.96e-04 & 8.46e-15 &  3.07  \\  
	\multicolumn{1}{c|}{PenC} & -9.610213e+00 & 2326.1 & 9.93e-04 & 8.43e-15 &  3.26  \\  
	\multicolumn{1}{c|}{SLPG} & -9.610306e+00 & 1848.4 & 9.46e-04 & 8.27e-15 &  2.45  \\  
	\multicolumn{1}{c|}{RCG} & -9.613055e+00 & 2417.4 & 9.92e-04 & 3.84e-15 &  12.39  \\ \hline 
	\multicolumn{6}{c}{$n= 500, p= 50$}    \\ \midrule 
	\multicolumn{1}{c|}{ExPen-CG} & -9.570277e+00 & 859.4 & 7.45e-04 & 1.69e-14 &  6.97  \\  
	\multicolumn{1}{c|}{PCAL} & -9.568621e+00 & 1917.1 & 9.90e-04 & 8.20e-15 &  6.81  \\  
	\multicolumn{1}{c|}{PenC} & -9.568620e+00 & 2176.8 & 9.90e-04 & 8.29e-15 &  7.61  \\  
	\multicolumn{1}{c|}{SLPG} & -9.568844e+00 & 1844.6 & 9.44e-04 & 8.11e-15 &  6.15  \\  
	\multicolumn{1}{c|}{RCG} & -9.571521e+00 & 2172.6 & 9.96e-04 & 3.15e-15 &  24.76  \\ \hline 
	\multicolumn{6}{c}{$n= 1000, p= 50$}    \\ \midrule 
	\multicolumn{1}{c|}{ExPen-CG} & -1.056750e+01 & 816.8 & 7.55e-04 & 1.69e-14 &  9.17  \\  
	\multicolumn{1}{c|}{PCAL} & -1.056576e+01 & 1789.7 & 9.91e-04 & 8.22e-15 &  10.03  \\  
	\multicolumn{1}{c|}{PenC} & -1.056577e+01 & 2190.5 & 9.90e-04 & 8.20e-15 &  11.79  \\  
	\multicolumn{1}{c|}{SLPG} & -1.056600e+01 & 1637.0 & 9.48e-04 & 7.97e-15 &  10.81  \\  
	\multicolumn{1}{c|}{RCG} & -1.056893e+01 & 3078.1 & 9.94e-04 & 3.07e-15 &  49.66  \\ \hline 
	\multicolumn{6}{c}{$n= 1500, p= 50$}    \\ \midrule 
	\multicolumn{1}{c|}{ExPen-CG} & -1.059979e+01 & 723.9 & 7.62e-04 & 1.63e-14 &  12.14  \\  
	\multicolumn{1}{c|}{PCAL} & -1.059830e+01 & 1632.3 & 9.87e-04 & 7.80e-15 &  14.00  \\  
	\multicolumn{1}{c|}{PenC} & -1.059828e+01 & 2004.8 & 9.96e-04 & 7.87e-15 &  16.55  \\  
	\multicolumn{1}{c|}{SLPG} & -1.059850e+01 & 1518.6 & 9.28e-04 & 7.90e-15 &  12.49  \\  
	\multicolumn{1}{c|}{RCG} & -1.060131e+01 & 2845.3 & 9.96e-04 & 2.79e-15 &  66.11  \\ \hline 
	\multicolumn{6}{c}{$n= 2000, p= 50$}    \\ \midrule 
	\multicolumn{1}{c|}{ExPen-CG} & -1.098884e+01 & 712.2 & 7.63e-04 & 1.67e-14 &  16.49  \\  
	\multicolumn{1}{c|}{PCAL} & -1.098736e+01 & 1584.2 & 9.81e-04 & 7.97e-15 &  19.29  \\  
	\multicolumn{1}{c|}{PenC} & -1.098741e+01 & 1837.1 & 9.67e-04 & 7.93e-15 &  21.78  \\  
	\multicolumn{1}{c|}{SLPG} & -1.098743e+01 & 1514.2 & 9.61e-04 & 7.80e-15 &  15.59  \\  
	\multicolumn{1}{c|}{RCG} & -1.099041e+01 & 3106.6 & 9.94e-04 & 2.61e-15 &  101.91  \\ \hline 
	\bottomrule[.4mm] 
	
	\caption{The results of the Brockett function minimization problems with varying $n$.}
	\label{Table_broc_n}
\end{longtable}

\begin{longtable}{@{}cccccc@{}}
	\toprule
	Solver                        & Fval    & Iteration & Stationarity & Feasibility & CPU time(s) \\* \midrule
	\endfirsthead
	\multicolumn{6}{c}%
	{{\bfseries Table \thetable\ continued from previous page}} \\
	\endhead
	\bottomrule
	\endfoot
	\endlastfoot
	\multicolumn{6}{c}{$n= 1000, p= 10$}    \\ \midrule 
	\multicolumn{1}{c|}{ExPen-CG} & -2.551708e+00 & 343.1 & 6.20e-04 & 4.62e-15 &  0.86  \\  
	\multicolumn{1}{c|}{PCAL} & -2.551070e+00 & 623.4 & 9.50e-04 & 2.19e-15 &  0.98  \\  
	\multicolumn{1}{c|}{PenC} & -2.551053e+00 & 675.0 & 9.78e-04 & 1.90e-15 &  1.01  \\  
	\multicolumn{1}{c|}{SLPG} & -2.551101e+00 & 582.3 & 9.41e-04 & 2.21e-15 &  0.87  \\  
	\multicolumn{1}{c|}{RCG} & -2.552008e+00 & 1278.9 & 9.91e-04 & 1.09e-15 &  5.69  \\ \hline 
	\multicolumn{6}{c}{$n= 1000, p= 30$}    \\ \midrule 
	\multicolumn{1}{c|}{ExPen-CG} & -6.479578e+00 & 576.9 & 6.95e-04 & 9.92e-15 &  3.16  \\  
	\multicolumn{1}{c|}{PCAL} & -6.478547e+00 & 1345.6 & 9.89e-04 & 4.87e-15 &  3.99  \\  
	\multicolumn{1}{c|}{PenC} & -6.478539e+00 & 1553.4 & 9.94e-04 & 5.10e-15 &  4.42  \\  
	\multicolumn{1}{c|}{SLPG} & -6.478541e+00 & 1289.7 & 9.72e-04 & 4.86e-15 &  4.81  \\  
	\multicolumn{1}{c|}{RCG} & -6.480546e+00 & 2474.2 & 9.95e-04 & 2.11e-15 &  22.56  \\ \hline 
	\multicolumn{6}{c}{$n= 1000, p= 50$}    \\ \midrule 
	\multicolumn{1}{c|}{ExPen-CG} & -1.032617e+01 & 838.2 & 7.45e-04 & 1.64e-14 &  9.00  \\  
	\multicolumn{1}{c|}{PCAL} & -1.032438e+01 & 1698.3 & 9.94e-04 & 8.25e-15 &  9.66  \\  
	\multicolumn{1}{c|}{PenC} & -1.032442e+01 & 2127.4 & 9.88e-04 & 8.22e-15 &  11.58  \\  
	\multicolumn{1}{c|}{SLPG} & -1.032457e+01 & 1522.5 & 9.39e-04 & 8.13e-15 &  9.71  \\  
	\multicolumn{1}{c|}{RCG} & -1.032752e+01 & 2895.4 & 9.96e-04 & 3.10e-15 &  49.96  \\ \hline 
	\multicolumn{6}{c}{$n= 1000, p= 70$}    \\ \midrule 
	\multicolumn{1}{c|}{ExPen-CG} & -1.312272e+01 & 850.7 & 7.84e-04 & 2.02e-14 &  12.65  \\  
	\multicolumn{1}{c|}{PCAL} & -1.312096e+01 & 2131.2 & 9.96e-04 & 1.03e-14 &  14.91  \\  
	\multicolumn{1}{c|}{PenC} & -1.312099e+01 & 2492.5 & 9.97e-04 & 1.03e-14 &  16.79  \\  
	\multicolumn{1}{c|}{SLPG} & -1.312106e+01 & 1848.2 & 9.54e-04 & 1.01e-14 &  16.20  \\  
	\multicolumn{1}{c|}{RCG} & -1.312455e+01 & 3298.1 & 9.96e-04 & 3.74e-15 &  69.18  \\ \hline 
	\multicolumn{6}{c}{$n= 1000, p= 100$}    \\ \midrule 
	\multicolumn{1}{c|}{ExPen-CG} & -2.039827e+01 & 1169.3 & 7.91e-04 & 2.69e-14 &  22.14  \\  
	\multicolumn{1}{c|}{PCAL} & -2.039602e+01 & 2413.2 & 9.97e-04 & 1.31e-14 &  24.27  \\  
	\multicolumn{1}{c|}{PenC} & -2.039603e+01 & 2838.3 & 9.94e-04 & 1.31e-14 &  27.86  \\  
	\multicolumn{1}{c|}{SLPG} & -2.039626e+01 & 2289.9 & 9.43e-04 & 1.29e-14 &  24.60  \\  
	\multicolumn{1}{c|}{RCG} & -2.040059e+01 & 4261.6 & 9.96e-04 & 4.85e-15 &  128.97  \\ \hline 
	\bottomrule[.4mm] 
	
	\caption{The results of the Brockett function minimization problems with varying $p$.}
	\label{Table_broc_p}
\end{longtable}

\section{Conclusion}
The optimization over the Stiefel manifold has a close connection with unconstrained optimization. To efficiently extend existing unconstrained optimization approaches to their Stiefel versions and establish the corresponding theoretical analysis, 
most existing approaches are mainly based on the frameworks summarized in \cite{Absil2009optimization}.
These approaches always involve computing the retractions and parallel/vector transports. However, computing retractions or parallel transport on the Stiefel manifold lack
efficiency or scalability, while computing the vector transport can hardly inherit nice techniques in theoretical analysis.

In this paper, we present a novel exact smooth penalty function and its corresponding penalty model \ref{Prob_Pen} for \ref{Prob_Ori}. We show that \ref{Prob_Pen} is well-defined under mild assumptions and study its theoretical properties. As illustrated in Figure \ref{Fig_roadmap_FOSP} and Figure \ref{Fig_roadmap_SOSP}, we have proved the first-order and second-order relationships between \ref{Prob_Ori} and \ref{Prob_Pen}, respectively. These properties guarantee that \ref{Prob_Pen} and \ref{Prob_Ori} share first-order or second-order stationary points or local minimizers with a sufficiently large given penalty parameter.

In conclusion, we can directly adopt unconstrained optimization approaches to solve \ref{Prob_Ori} through the bridge built by \ref{Prob_Pen}. Meanwhile, we can easily inherit the nice convergence properties of those approaches. We use the nonlinear conjugate gradient method as an instance. We present its \ref{Prob_Pen} version and establish its global convergence, worst-case complexity and the ability to escaping the saddle point. It is worth mentioning that this \ref{Prob_Pen} version is performed in Euclidean space and hence avoids computing the retractions or parallel transports on the Stiefel manifold. Our present example highlights that those progress in nonconvex unconstrained optimization will immediately benefit optimization over the Stiefel manifold through \ref{Prob_Pen}. Moreover, the presented numerical examples further address the great potential of \ref{Prob_Pen}.

	\bibliography{ref}

\begin{thebibliography}{71}
\providecommand{\natexlab}[1]{#1}
\providecommand{\url}[1]{\texttt{#1}}
\expandafter\ifx\csname urlstyle\endcsname\relax
  \providecommand{\doi}[1]{doi: #1}\else
  \providecommand{\doi}{doi: \begingroup \urlstyle{rm}\Url}\fi

\bibitem[Ablin and Peyr{\'e}(2022)]{ablin2022fast}
Pierre Ablin and Gabriel Peyr{\'e}.
\newblock Fast and accurate optimization on the orthogonal manifold without
  retraction.
\newblock In \emph{International Conference on Artificial Intelligence and
  Statistics}, pages 5636--5657. PMLR, 2022.

\bibitem[Abrudan et~al.(2009)Abrudan, Eriksson, and
  Koivunen]{abrudan2009conjugate}
Traian Abrudan, Jan Eriksson, and Visa Koivunen.
\newblock Conjugate gradient algorithm for optimization under unitary matrix
  constraint.
\newblock \emph{Signal Processing}, 89\penalty0 (9):\penalty0 1704--1714, 2009.

\bibitem[Abrudan et~al.(2008)Abrudan, Eriksson, and
  Koivunen]{abrudan2008steepest}
Traian~E Abrudan, Jan Eriksson, and Visa Koivunen.
\newblock Steepest descent algorithms for optimization under unitary matrix
  constraint.
\newblock \emph{IEEE Transactions on Signal Processing}, 56\penalty0
  (3):\penalty0 1134--1147, 2008.

\bibitem[Absil et~al.(2009)Absil, Mahony, and Sepulchre]{Absil2009optimization}
P-A Absil, Robert Mahony, and Rodolphe Sepulchre.
\newblock \emph{Optimization algorithms on matrix manifolds}.
\newblock Princeton University Press, 2009.

\bibitem[Ahn and Sra(2020)]{ahn2020nesterov}
Kwangjun Ahn and Suvrit Sra.
\newblock From nesterov’s estimate sequence to riemannian acceleration.
\newblock In \emph{Conference on Learning Theory}, pages 84--118. PMLR, 2020.

\bibitem[Bansal et~al.(2018)Bansal, Chen, and Wang]{bansal2018can}
Nitin Bansal, Xiaohan Chen, and Zhangyang Wang.
\newblock Can we gain more from orthogonality regularizations in training deep
  networks?
\newblock In \emph{Advances in Neural Information Processing Systems}, pages
  4261--4271, 2018.

\bibitem[B{\'e}cigneul and Ganea(2018)]{becigneul2018riemannian}
Gary B{\'e}cigneul and Octavian-Eugen Ganea.
\newblock Riemannian adaptive optimization methods.
\newblock \emph{arXiv preprint arXiv:1810.00760}, 2018.

\bibitem[Bertsekas(2014)]{bertsekas2014constrained}
Dimitri~P Bertsekas.
\newblock \emph{Constrained optimization and Lagrange multiplier methods}.
\newblock Academic press, 2014.

\bibitem[Bolte et~al.(2014)Bolte, Sabach, and Teboulle]{bolte2014proximal}
J{\'e}r{\^o}me Bolte, Shoham Sabach, and Marc Teboulle.
\newblock Proximal alternating linearized minimization for nonconvex and
  nonsmooth problems.
\newblock \emph{Mathematical Programming}, 146\penalty0 (1):\penalty0 459--494,
  2014.

\bibitem[Boumal(2020)]{boumal2020introduction}
Nicolas Boumal.
\newblock An introduction to optimization on smooth manifolds.
\newblock \emph{Available online, May}, 2020.

\bibitem[Cai et~al.(2018)Cai, Zhang, Bai, and Li]{cai2018eigenvector}
Yunfeng Cai, Lei-Hong Zhang, Zhaojun Bai, and Ren-Cang Li.
\newblock On an eigenvector-dependent nonlinear eigenvalue problem.
\newblock \emph{SIAM Journal on Matrix Analysis and Applications}, 39\penalty0
  (3):\penalty0 1360--1382, 2018.

\bibitem[Criscitiello and Boumal(2019)]{criscitiello2019efficiently}
Chris Criscitiello and Nicolas Boumal.
\newblock Efficiently escaping saddle points on manifolds.
\newblock \emph{arXiv preprint arXiv:1906.04321}, 2019.

\bibitem[Criscitiello and Boumal(2020)]{criscitiello2020accelerated}
Chris Criscitiello and Nicolas Boumal.
\newblock An accelerated first-order method for non-convex optimization on
  manifolds.
\newblock \emph{arXiv preprint arXiv:2008.02252}, 2020.

\bibitem[Dai and Yuan(1999)]{dai1999nonlinear}
Yu-Hong Dai and Yaxiang Yuan.
\newblock A nonlinear conjugate gradient method with a strong global
  convergence property.
\newblock \emph{SIAM Journal on optimization}, 10\penalty0 (1):\penalty0
  177--182, 1999.

\bibitem[Daniel(1967)]{daniel1967conjugate}
James~W Daniel.
\newblock The conjugate gradient method for linear and nonlinear operator
  equations.
\newblock \emph{SIAM Journal on Numerical Analysis}, 4\penalty0 (1):\penalty0
  10--26, 1967.

\bibitem[Figueroa and Dalmau(2020)]{figueroa2020transportless}
Edgar~Fuentes Figueroa and Oscar Dalmau.
\newblock Transportless conjugate gradient for optimization on stiefel
  manifold.
\newblock \emph{Computational and Applied Mathematics}, 39:\penalty0 1--18,
  2020.

\bibitem[Fischler and Bolles(1981)]{fischler1981random}
Martin~A Fischler and Robert~C Bolles.
\newblock Random sample consensus: a paradigm for model fitting with
  applications to image analysis and automated cartography.
\newblock \emph{Communications of the ACM}, 24\penalty0 (6):\penalty0 381--395,
  1981.

\bibitem[Fletcher and Reeves(1964)]{fletcher1964function}
Reeves Fletcher and Colin~M Reeves.
\newblock Function minimization by conjugate gradients.
\newblock \emph{The computer journal}, 7\penalty0 (2):\penalty0 149--154, 1964.

\bibitem[Fletcher(2013)]{fletcher2013practical}
Roger Fletcher.
\newblock \emph{Practical methods of optimization}.
\newblock John Wiley \& Sons, 2013.

\bibitem[Gao et~al.(2018)Gao, Liu, Chen, and Yuan]{Gao2018New}
Bin Gao, Xin Liu, Xiaojun Chen, and Ya-xiang Yuan.
\newblock A new first-order algorithmic framework for optimization problems
  with orthogonality constraints.
\newblock \emph{SIAM Journal on Optimization}, 28\penalty0 (1):\penalty0
  302--332, 2018.

\bibitem[Gao et~al.(2019)Gao, Liu, and Yuan]{gao2019parallelizable}
Bin Gao, Xin Liu, and Ya-xiang Yuan.
\newblock Parallelizable algorithms for optimization problems with
  orthogonality constraints.
\newblock \emph{SIAM Journal on Scientific Computing}, 41\penalty0
  (3):\penalty0 A1949--A1983, 2019.

\bibitem[Hager and Zhang(2005)]{hager2005new}
William~W Hager and Hongchao Zhang.
\newblock A new conjugate gradient method with guaranteed descent and an
  efficient line search.
\newblock \emph{SIAM Journal on optimization}, 16\penalty0 (1):\penalty0
  170--192, 2005.

\bibitem[Hager and Zhang(2006)]{hager2006survey}
William~W Hager and Hongchao Zhang.
\newblock A survey of nonlinear conjugate gradient methods.
\newblock \emph{Pacific journal of Optimization}, 2\penalty0 (1):\penalty0
  35--58, 2006.

\bibitem[Han and Gao(2020)]{han2020escape}
Andi Han and Junbin Gao.
\newblock Escape saddle points faster on manifolds via perturbed riemannian
  stochastic recursive gradient.
\newblock \emph{arXiv preprint arXiv:2010.12191}, 2020.

\bibitem[Han and Kim(2015)]{han2015unsupervised}
Dongyoon Han and Junmo Kim.
\newblock Unsupervised simultaneous orthogonal basis clustering feature
  selection.
\newblock In \emph{Proceedings of the IEEE conference on computer vision and
  pattern recognition}, pages 5016--5023, 2015.

\bibitem[Harris et~al.(2020)Harris, Millman, van~der Walt, Gommers, Virtanen,
  Cournapeau, Wieser, Taylor, Berg, Smith, Kern, Picus, Hoyer, van Kerkwijk,
  Brett, Haldane, del R{\'{i}}o, Wiebe, Peterson, G{\'{e}}rard-Marchant,
  Sheppard, Reddy, Weckesser, Abbasi, Gohlke, and Oliphant]{numpy2020array}
Charles~R. Harris, K.~Jarrod Millman, St{\'{e}}fan~J. van~der Walt, Ralf
  Gommers, Pauli Virtanen, David Cournapeau, Eric Wieser, Julian Taylor,
  Sebastian Berg, Nathaniel~J. Smith, Robert Kern, Matti Picus, Stephan Hoyer,
  Marten~H. van Kerkwijk, Matthew Brett, Allan Haldane, Jaime~Fern{\'{a}}ndez
  del R{\'{i}}o, Mark Wiebe, Pearu Peterson, Pierre G{\'{e}}rard-Marchant,
  Kevin Sheppard, Tyler Reddy, Warren Weckesser, Hameer Abbasi, Christoph
  Gohlke, and Travis~E. Oliphant.
\newblock Array programming with {NumPy}.
\newblock \emph{Nature}, 585\penalty0 (7825):\penalty0 357--362, September
  2020.
\newblock \doi{10.1038/s41586-020-2649-2}.
\newblock URL \url{https://doi.org/10.1038/s41586-020-2649-2}.

\bibitem[Helfrich et~al.(2018)Helfrich, Willmott, and
  Ye]{helfrich2018orthogonal}
Kyle Helfrich, Devin Willmott, and Qiang Ye.
\newblock Orthogonal recurrent neural networks with scaled cayley transform.
\newblock In \emph{International Conference on Machine Learning}, pages
  1969--1978. PMLR, 2018.

\bibitem[Hestenes(1969)]{hestenes1969multiplier}
Magnus~R Hestenes.
\newblock Multiplier and gradient methods.
\newblock \emph{Journal of optimization theory and applications}, 4\penalty0
  (5):\penalty0 303--320, 1969.

\bibitem[Hestenes et~al.(1952)Hestenes, Stiefel, et~al.]{hestenes1952methods}
Magnus~Rudolph Hestenes, Eduard Stiefel, et~al.
\newblock \emph{Methods of conjugate gradients for solving linear systems},
  volume~49.
\newblock NBS Washington, DC, 1952.

\bibitem[Higham and Papadimitriou(1994)]{higham1994APA}
Nicholas~J. Higham and Pythagoras Papadimitriou.
\newblock A parallel algorithm for computing the polar decomposition.
\newblock \emph{Parallel Computing}, 20:\penalty0 1161--1173, 1994.

\bibitem[Hosseini(2015)]{hosseini2015convergence}
S~Hosseini.
\newblock Convergence of nonsmooth descent methods via kurdyka--lojasiewicz
  inequality on riemannian manifolds.
\newblock \emph{Hausdorff Center for Mathematics and Institute for Numerical
  Simulation, University of Bonn (2015,(INS Preprint No. 1523))}, 2015.

\bibitem[Hou et~al.(2020)Hou, Li, and Zhang]{hou2020analysis}
Thomas~Y Hou, Zhenzhen Li, and Ziyun Zhang.
\newblock Analysis of asymptotic escape of strict saddle sets in manifold
  optimization.
\newblock \emph{SIAM Journal on Mathematics of Data Science}, 2\penalty0
  (3):\penalty0 840--871, 2020.

\bibitem[Hu et~al.(2020)Hu, Liu, Wen, and Yuan]{hu2020brief}
Jiang Hu, Xin Liu, Zai-Wen Wen, and Ya-Xiang Yuan.
\newblock A brief introduction to manifold optimization.
\newblock \emph{Journal of the Operations Research Society of China},
  8\penalty0 (2):\penalty0 199--248, 2020.

\bibitem[{Hu} and {Liu}(2020)]{hu2020anefficiency}
Xiaoyin {Hu} and Xin {Liu}.
\newblock An efficient orthonormalization-free approach for sparse dictionary
  learning and dual principal component pursuit.
\newblock \emph{Sensors}, 20\penalty0 (3041), 2020.

\bibitem[Huang and Wei(2019)]{huang2019riemannian}
Wen Huang and Ke~Wei.
\newblock Riemannian proximal gradient methods.
\newblock \emph{arXiv preprint arXiv:1909.06065}, 2019.

\bibitem[Jiang and Dai(2015)]{jiang2015framework}
Bo~Jiang and Yu-Hong Dai.
\newblock A framework of constraint preserving update schemes for optimization
  on stiefel manifold.
\newblock \emph{Mathematical Programming}, 153\penalty0 (2):\penalty0 535--575,
  2015.

\bibitem[Kohn and Sham(1965)]{Kohn1965Self}
Walter Kohn and Lu~Jeu Sham.
\newblock Self-consistent equations including exchange and correlation effects.
\newblock \emph{Physical review}, 140\penalty0 (4A):\penalty0 A1133, 1965.

\bibitem[Lezcano-Casado(2019)]{lezcano2019trivializations}
Mario Lezcano-Casado.
\newblock Trivializations for gradient-based optimization on manifolds.
\newblock \emph{arXiv preprint arXiv:1909.09501}, 2019.

\bibitem[Lezcano-Casado and Mart{\i}nez-Rubio(2019)]{lezcano2019cheap}
Mario Lezcano-Casado and David Mart{\i}nez-Rubio.
\newblock Cheap orthogonal constraints in neural networks: A simple
  parametrization of the orthogonal and unitary group.
\newblock In \emph{International Conference on Machine Learning}, pages
  3794--3803. PMLR, 2019.

\bibitem[Li et~al.(2020)Li, Fuxin, and Todorovic]{li2020efficient}
Jun Li, Li~Fuxin, and Sinisa Todorovic.
\newblock Efficient riemannian optimization on the stiefel manifold via the
  cayley transform.
\newblock \emph{arXiv preprint arXiv:2002.01113}, 2020.

\bibitem[Li et~al.(2018)Li, Zhang, Zhang, Liu, and Nie]{li2018generalized}
Xuelong Li, Han Zhang, Rui Zhang, Yun Liu, and Feiping Nie.
\newblock Generalized uncorrelated regression with adaptive graph for
  unsupervised feature selection.
\newblock \emph{IEEE transactions on neural networks and learning systems},
  30\penalty0 (5):\penalty0 1587--1595, 2018.

\bibitem[Liu et~al.(2014)Liu, Wang, Wen, and Yuan]{liu2014Convergence}
Xin Liu, Xiao Wang, Zaiwen Wen, and Yaxiang Yuan.
\newblock On the convergence of the self-consistent field iteration in
  kohn--sham density functional theory.
\newblock \emph{SIAM Journal on Matrix Analysis and Applications}, 35\penalty0
  (2):\penalty0 546--558, 2014.

\bibitem[Liu et~al.(2015)Liu, Wen, Wang, Ulbrich, and Yuan]{liu2015analysis}
Xin Liu, Zaiwen Wen, Xiao Wang, Michael Ulbrich, and Yaxiang Yuan.
\newblock On the analysis of the discretized kohn--sham density functional
  theory.
\newblock \emph{SIAM Journal on Numerical Analysis}, 53\penalty0 (4):\penalty0
  1758--1785, 2015.

\bibitem[Liu and Storey(1991)]{liu1991efficient}
Y~Liu and C~Storey.
\newblock Efficient generalized conjugate gradient algorithms, part 1: theory.
\newblock \emph{Journal of optimization theory and applications}, 69\penalty0
  (1):\penalty0 129--137, 1991.

\bibitem[{\L}ojasiewicz(1961)]{lojasiewicz1961probleme}
Stanis{\l}aw {\L}ojasiewicz.
\newblock Sur le probleme de la division.
\newblock 1961.

\bibitem[Lojasiewicz(1963)]{lojasiewicz1963propriete}
Stanislaw Lojasiewicz.
\newblock Une propri{\'e}t{\'e} topologique des sous-ensembles analytiques
  r{\'e}els.
\newblock \emph{Les {\'e}quations aux d{\'e}riv{\'e}es partielles},
  117:\penalty0 87--89, 1963.

\bibitem[Maduranga et~al.(2019)Maduranga, Helfrich, and
  Ye]{maduranga2019complex}
Kehelwala~DG Maduranga, Kyle~E Helfrich, and Qiang Ye.
\newblock Complex unitary recurrent neural networks using scaled cayley
  transform.
\newblock In \emph{Proceedings of the AAAI Conference on Artificial
  Intelligence}, volume~33, pages 4528--4535, 2019.

\bibitem[Nocedal and Wright(2006)]{nocedal2006numerical}
Jorge Nocedal and Stephen Wright.
\newblock \emph{Numerical optimization}.
\newblock Springer Science \& Business Media, 2006.

\bibitem[Pearson(1901)]{pearson1901liii}
Karl Pearson.
\newblock Liii. on lines and planes of closest fit to systems of points in
  space.
\newblock \emph{The London, Edinburgh, and Dublin Philosophical Magazine and
  Journal of Science}, 2\penalty0 (11):\penalty0 559--572, 1901.

\bibitem[Polyak(1969)]{polyak1969conjugate}
Boris~Teodorovich Polyak.
\newblock The conjugate gradient method in extremal problems.
\newblock \emph{USSR Computational Mathematics and Mathematical Physics},
  9\penalty0 (4):\penalty0 94--112, 1969.

\bibitem[Powell(1969)]{powell1969method}
Michael~JD Powell.
\newblock A method for nonlinear constraints in minimization problems.
\newblock \emph{Optimization}, pages 283--298, 1969.

\bibitem[Sato(2016)]{sato2016dai}
Hiroyuki Sato.
\newblock A dai--yuan-type riemannian conjugate gradient method with the weak
  wolfe conditions.
\newblock \emph{Computational optimization and Applications}, 64\penalty0
  (1):\penalty0 101--118, 2016.

\bibitem[Sato and Iwai(2015)]{sato2015new}
Hiroyuki Sato and Toshihiro Iwai.
\newblock A new, globally convergent riemannian conjugate gradient method.
\newblock \emph{Optimization}, 64\penalty0 (4):\penalty0 1011--1031, 2015.

\bibitem[Siegel(2019)]{siegel2019accelerated}
Jonathan~W Siegel.
\newblock Accelerated optimization with orthogonality constraints.
\newblock \emph{arXiv preprint arXiv:1903.05204}, 2019.

\bibitem[Sun et~al.(2019)Sun, Flammarion, and Fazel]{sun2019escaping}
Yue Sun, Nicolas Flammarion, and Maryam Fazel.
\newblock Escaping from saddle points on riemannian manifolds.
\newblock \emph{arXiv preprint arXiv:1906.07355}, 2019.

\bibitem[Townsend et~al.(2016)Townsend, Koep, and
  Weichwald]{townsend2016pymanopt}
James Townsend, Niklas Koep, and Sebastian Weichwald.
\newblock Pymanopt: A python toolbox for optimization on manifolds using
  automatic differentiation.
\newblock \emph{arXiv preprint arXiv:1603.03236}, 2016.

\bibitem[Virtanen et~al.(2020)Virtanen, Gommers, Oliphant, Haberland, Reddy,
  Cournapeau, Burovski, Peterson, Weckesser, Bright,
  et~al.]{scipy2020SciPyNMeth}
Pauli Virtanen, Ralf Gommers, Travis~E Oliphant, Matt Haberland, Tyler Reddy,
  David Cournapeau, Evgeni Burovski, Pearu Peterson, Warren Weckesser, Jonathan
  Bright, et~al.
\newblock Scipy 1.0: fundamental algorithms for scientific computing in python.
\newblock \emph{Nature methods}, 17\penalty0 (3):\penalty0 261--272, 2020.

\bibitem[Wang et~al.(2020{\natexlab{a}})Wang, Gao, and
  Liu]{wang2020multipliers}
Lei Wang, Bin Gao, and Xin Liu.
\newblock Multipliers correction methods for optimization problems over the
  stiefel manifold.
\newblock \emph{arXiv preprint arXiv:2011.14781}, 2020{\natexlab{a}}.

\bibitem[Wang et~al.(2020{\natexlab{b}})Wang, Liu, and
  Zhang]{wang2020distributed}
Lei Wang, Xin Liu, and Yin Zhang.
\newblock A distributed and secure algorithm for computing dominant svd based
  on projection splitting.
\newblock \emph{arXiv preprint arXiv:2012.03461}, 2020{\natexlab{b}}.

\bibitem[Wen and Yin(2013)]{wen2013feasible}
Zaiwen Wen and Wotao Yin.
\newblock A feasible method for optimization with orthogonality constraints.
\newblock \emph{Mathematical Programming}, 142\penalty0 (1-2):\penalty0
  397--434, 2013.

\bibitem[Xiao et~al.(2020{\natexlab{a}})Xiao, Liu, and Yuan]{xiao2020class}
Nachuan Xiao, Xin Liu, and Ya-xiang Yuan.
\newblock A class of smooth exact penalty function methods for optimization
  problems with orthogonality constraints.
\newblock \emph{Optimization Methods and Software}, pages 1--37,
  2020{\natexlab{a}}.

\bibitem[Xiao et~al.(2020{\natexlab{b}})Xiao, Liu, and Yuan]{xiao2020l21}
Nachuan Xiao, Xin Liu, and Ya-xiang Yuan.
\newblock Exact penalty function for $\ell_{2,1}$ norm minimization with
  orthogonality constraints.
\newblock \emph{optimization online preprint}, 2020{\natexlab{b}}.

\bibitem[Xiao et~al.(2021)Xiao, Liu, and Yuan]{xiao2021penalty}
Nachuan Xiao, Xin Liu, and Ya-xiang Yuan.
\newblock A penalty-free infeasible approach for a class of nonsmooth
  opimtization problems over the stiefel manifold.
\newblock \emph{arXiv preprint arXiv:2103.03514}, 2021.

\bibitem[Xiao et~al.(2022)Xiao, Wang, Gao, Liu, and Yuan]{stop2022website}
Nachuan Xiao, Lei Wang, Bin Gao, Xin Liu, and Ya-xiang Yuan.
\newblock {STOP}: a toolbox for {S}tiefel manifold optimization, 2022.
\newblock URL \url{https://stmopt.gitee.io/}.

\bibitem[Yang et~al.(2009)Yang, Gao, and Meza]{yang2009Convergence}
Chao Yang, Weiguo Gao, and Juan~C Meza.
\newblock On the convergence of the self-consistent field iteration for a class
  of nonlinear eigenvalue problems.
\newblock \emph{SIAM Journal on Matrix Analysis and Applications}, 30\penalty0
  (4):\penalty0 1773--1788, 2009.

\bibitem[Zhang and Sra(2018)]{zhang2018towards}
Hongyi Zhang and Suvrit Sra.
\newblock Towards riemannian accelerated gradient methods.
\newblock \emph{arXiv preprint arXiv:1806.02812}, 2018.

\bibitem[Zhang et~al.(2018{\natexlab{a}})Zhang, Zhang, and Sra]{zhang2018r}
Jingzhao Zhang, Hongyi Zhang, and Suvrit Sra.
\newblock R-spider: A fast riemannian stochastic optimization algorithm with
  curvature independent rate.
\newblock \emph{arXiv preprint arXiv:1811.04194}, 2018{\natexlab{a}}.

\bibitem[Zhang et~al.(2018{\natexlab{b}})Zhang, Nie, and Li]{zhang2018feature}
Rui Zhang, Feiping Nie, and Xuelong Li.
\newblock Feature selection under regularized orthogonal least square
  regression with optimal scaling.
\newblock \emph{Neurocomputing}, 273:\penalty0 547--553, 2018{\natexlab{b}}.

\bibitem[Zhou et~al.(2019)Zhou, Yuan, and Feng]{zhou2019faster}
Pan Zhou, Xiao-Tong Yuan, and Jiashi Feng.
\newblock Faster first-order methods for stochastic non-convex optimization on
  riemannian manifolds.
\newblock In \emph{The 22nd International Conference on Artificial Intelligence
  and Statistics}, pages 138--147. PMLR, 2019.

\bibitem[Zhu(2017)]{zhu2017riemannian}
Xiaojing Zhu.
\newblock A riemannian conjugate gradient method for optimization on the
  stiefel manifold.
\newblock \emph{Computational optimization and Applications}, 67\penalty0
  (1):\penalty0 73--110, 2017.

\bibitem[Zoutendijk(1970)]{zoutendijk1970nonlinear}
G~Zoutendijk.
\newblock Nonlinear programming, computational methods.
\newblock \emph{Integer and nonlinear programming}, pages 37--86, 1970.

\end{thebibliography}
	\bibliographystyle{plainnat}
	
\end{document}